\documentclass{amsart}

\newcommand{\RN}[1]{\textup{\uppercase\expandafter{\romannumeral#1}}}
\usepackage{amsmath}
\usepackage{amssymb}
\usepackage{enumerate}
\usepackage[hidelinks]{hyperref}
\usepackage{mathtools}
\usepackage[capitalise]{cleveref}
\numberwithin{equation}{section}

\newcommand\mc{\mathcal}
\newcommand\mf{\mathfrak}
\newcommand\mb{\mathbb}
\newcommand\mbf{\mathbf}
\crefname{equation}{}{}
\newtheorem{theorem}{Theorem}[section]
\newtheorem{lemma}[theorem]{Lemma}
\newtheorem{proposition}[theorem]{Proposition}
\newtheorem{corollary}[theorem]{Corollary}
\title{A Generalized vanishing theorem for Blow-ups of Quasi-smooth Stacks}
\author{Yu Zhao}
\theoremstyle{definition}

\newtheorem{conjecture}[theorem]{Conjecture}
\newtheorem{definition}[theorem]{Definition}
\newtheorem{assumption}[theorem]{Assumption}
\newtheorem{example}[theorem]{Example}

\theoremstyle{remark}
\newtheorem{remark}[theorem]{Remark}
\usepackage{epigraph}
\usepackage{comment}
\usepackage{tikz}
\usepackage{tikz-cd}

\usepackage{enumitem}

\begin{document}
\maketitle
\begin{center}
  {\it To our friend Kun Wang (1993-2022).}
\end{center}

\begin{abstract}
  We prove a generalized vanishing theorem for certain quasi-coherent sheaves along the derived blow-ups of quasi-smooth derived Artin stacks. We give four applications of the generalized vanishing theorem: we prove a $K$-theoretic version of the generalized vanishing theorem which verified a conjecture of the author \cite{yuzhaoderived} and give a new proof of the $K$-theoretic localization theorem for quasi-smooth derived schemes through the intrinsic blow-up theory of Kiem-Li-Savvas \cite{kiem2017generalized}; we prove a desingularization theorem for quasi-smooth derived schemes and give an approximation formula for the virtual fundamental classes; we give a resolution of the diagonal along the projection map of blow-ups of smooth varieties, which strengthens the semi-orthogonal decomposition theorem of Orlov \cite{MR1208153}; we illustrate the relation between the generalized vanishing theorem and weak categorifications of quantum loop and toroidal algebra actions on the derived category of Nakajima quiver varieties. We also propose several conjectures related to birational geometry and the $L_{\infty}$-algebroid of Calaque-C\u{a}ld\u{a}raru-Tu \cite{MR3228441}.
\end{abstract}
\setlength{\epigraphwidth}{0.45\textwidth}
\epigraph{\itshape Struggling on the edge of the youth;\\ Gazing from the end of freedom; \\Wandering in the barren prairie;\\ Looking forward to the ideal.}{Lao Wang,\textit{ I am still young}}
\setlength{\epigraphwidth}{0.45\textwidth}

\section{Introduction}
\subsection{The main result of this paper}

Given a closed embedding of derived stacks with a cotangent complex,
$$f:Z\to X,$$
Hekking and his collaborators \cite{Hekking_2022} defined a derived enhancement of the extended Rees algebra
$$R_{Z/X}^{ext},$$
and moreover defined a derived blow-up
$\mb{B}l_{Z/X}.$ The purpose of this paper is to establish a generalized vanishing theorem for certain quasi-coherent sheaves: we will compare the following two $\mb{Z}$-graded $\mc{O}_{X}[t^{-1}]$ quasi-coherent modules over $X$, where $t^{-1}$ has homogeneous degree $-1$:
\begin{enumerate}
\item the derived extended Rees algebra $R_{Z/X}^{ext}$ defined in Hekking's thesis \cite{hekking2021graded};
\item
  $$W_{Z/X}:=\bigoplus_{d\in \mb{Z}}pr_{Z/X*}\mc{O}_{\mb{B}l_{Z/X}}(d)$$
where $\mc{O}_{\mb{B}l_{Z/X}}$ is the exceptional divisor of the derived blow-up, and the $t^{-1}$ action is induced by the push-forward of the canonical map of exceptional divisors
$$\mc{O}_{\mb{B}l_{Z/X}}(d+1)\to\mc{O}_{\mb{B}l_{Z/X}}(d).$$
\end{enumerate}

The motivation to consider the generalized vanishing theorem is comprehensive, complicated, and not straightforward. Thus we will also give four applications and several conjectures to illustrate the relation between the generalized vanishing theorem and different areas of algebraic geometry and representation theory in a very rough way.
\subsection{Assumptions and the structure theorem}
We first state some assumptions and notations we need to formulate the generalized vanishing theorem:
\begin{assumption}
  \label{ass1}
  We propose the following assumptions for a closed embedding of $S$-derived stacks
  $$f:Z\to X:$$
  where $Z$ and $X$ are derived Artin stacks over $S$ with cotangent complexes $L_{Z/S}$ and $L_{X/S}$ which are both perfect and have $Tor$-amplitude $[-1,1]$ and constant ranks. Here $S$ is a derived Artin stack over an field $\mb{C}$ which is algebraically and of characteristic $0$.\end{assumption}
We make the following notations:
\begin{itemize}
\item $pr_{Z/X} \text{ or } pr_{f}:\mb{B}l_{Z/X}\to X$ as the projection map;
\item $\mc{O}_{\mb{B}l_{Z/X}}(1)$ as the (virtual) exceptional divisor of $\mb{B}l_{Z/X}$;
\item $C_{Z/X} \text{ or } C_{f}:=L_{Z/X}[-1]$ as the cotangent complex of $Z$ over $X$;
\item $r:=rank(L_{X/S})-rank(L_{Z/S})$ as the virtual codimension of $Z$ over $X$;
\item $\mb{P}_{Z}(\mc{F})$ is Jiang's derived projectivization in \cite{jiang2022derived}, where $\mc{F}$ is a quasi-coherent sheaf over $Z$.
\end{itemize}

\begin{theorem}[\cref{structurethm}, Structure theorem of derived blow-ups]
  Assuming \cref{ass1}, we have that
  \begin{enumerate}
      \item the blow-up $\mb{B}l_{Z/X}$ is also a quasi-smooth derived Artin stack over $S$;
      \item the cotangent complex $L_{\mb{B}l_{Z/X}/S}$ and $L_{X/S}$ have the same rank;
      \item the projection $pr_{Z/X}$ is proper.
  \end{enumerate}
\end{theorem}
\subsection{The generalized vanishing theorem \RN{1} and \RN{2}}
 We give two formulations of the generalized vanishing theorems. The first is more conceptual, and the second is more useful for applications.
\begin{theorem}[Generalized vanishing theorem \RN{1}, \cref{isomorphism2}]
 There exists a canonical morphism of $\mb{Z}$-graded quasi-coherent $\mc{O}_{X}[t^{-1}]$-modules
  $$\phi^{f}:R_{Z/X}^{ext}\to W_{Z/X}$$
  with the following commutative diagram (up to equivalence) where all the horizontal and vertical arrows are fiber sequences:
  \begin{equation*}
      \begin{tikzcd}
      R_{Z/X}^{ext}\ar{r}{t^{-1}} \ar{d}{\phi^{f}} & R_{Z/X}^{ext}\ar{r}\ar{d}{\phi^{f}} & \mf{H}_{\bullet}^{f+} \ar{d}{f_{*}\phi_{C_{Z/X}}} \\
      W_{Z/X}\ar{r}{t^{-1}}\ar{d} & W_{Z/X} \ar{d}\ar{r}& f_{*}pr_{C_{Z/X}*}(\mc{O}_{\mb{P}_{Z}(C_{Z/X})}(\bullet)) \ar{d}{f_{*}\psi_{C_{Z/X}}}\\
      cofib(\phi^{f})\ar{r}{t^{-1}} & cofib(\phi^{f})\ar{r} & \mf{H}_{\bullet}^{f-},
    \end{tikzcd}
  \end{equation*}
  where
  \begin{itemize}
  \item $cofib(\phi^{f})$ is the cofiber of $\phi^{f}$:
  \item we denote
  \begin{align*}
   &   \mf{H}_{b}^{f+}:=f_{*}S^{b}_{Z}(C_{Z/X})\\
   & \mf{H}_{b}^{f-}:=f_{*}(S^{-r-b}_{Z}(C_{Z/X})^{\vee}\otimes det(C_{Z/X})^{\vee})[1-r]
  \end{align*}
  for $b\in \mb{Z}$. We denote $\mf{H}_{\bullet}^{f+}:=\oplus_{b\in \mb{Z}}\mf{H}_{b}^{f+}$ and make similar notations for $pr_{C_{Z/X}*}(\mc{O}_{\mb{P}_{Z}(C_{Z/X})}(\bullet))$ and $\mf{H}_{\bullet}^{f-}$.
  \item $\phi_{C_{Z/X}}$ and $\psi_{C_{Z/X}}$ is defined in Jiang's generalized Serre theorem (\cref{thm:serre2}).
  \end{itemize}
  Moreover, when $d>-r$, the degree $d$ part of $\phi^{f}$, which we denote as $\phi^{f}_{\geq d}$, is an equivalence.
\end{theorem}
\begin{theorem}[Generalized vanishing theorem \RN{2}, \cref{thm:main}]
  \label{thm:nmain}
  There exists $\mf{W}_{a,b}^{f}\in \mathrm{QCoh}(X)$ for any $a\leq b$, with morphisms
  \begin{align*}
   \mf{w}_{(a,b),(a',b')}^{f}:\mf{W}^{f}_{a,b}\to \mf{W}^{f}_{a',b'},\quad  h_{a,b}^{f+}:\mf{W}_{a,b}^{f}\to \mf{H}_{a}^{f+},\quad h_{a,b}^{f-}:\mf{W}_{a,b}^{f}\to \mf{H}_{b}^{f-}
 \end{align*}
for $a'\leq a, b'\leq b$, such that
\begin{enumerate}
 \item 
  $$\mf{w}_{(a_{2},b_{2}),(a_{3},b_{3})}^{f}\circ\mf{w}_{(a_{1},b_{1}),(a_{2},b_{2})}^{f}\cong \mf{w}_{(a_{1},b_{1}),(a_{3},b_{3})}^{f}$$
  when all the above morphisms are well defined and $\mf{w}_{(a,b),(a,b)}^{f}\cong id$. 
  \item For any $a$, we have
    $$\mf{W}_{a,a}^{f}\cong pr_{f*}\mc{O}_{\mb{B}l_{Z/X}}(a)$$
    and for any $a\leq b$, the morphism $\mf{w}_{(b,b),(a,a)}^{f}$ is represented by 
    $$pr_{Z/X*}\mc{O}_{\mb{B}l_{Z/X}}(b)\xrightarrow{t^{a-b}} pr_{Z/X*}\mc{O}_{\mb{B}l_{Z/X}}(a).$$
\item For any $b\geq -r+1$, we have
    $$\mf{W}_{a,b}^{f}\cong R_{Z/X}^{a},$$
    and for any $b\geq b'\geq -r+1$, the morphism $\mf{w}_{(a,b),(a',b')}^{f}$ is represented by
    $$t^{a'-a}:R_{Z/X}^{a}\to R_{Z/X}^{a'}.$$
\item  For any $a<b$, we have the commutative diagram (up to equivalence) where all the rows and columns are fiber sequences:
\begin{equation*}
    \begin{tikzcd}[column sep=2cm]
        \mf{W}_{a+1,b+1}^{f}\ar{r}{\mf{w}_{(a+1,b+1),(a+1,b)}^{f}}\ar{d}{\mf{w}_{(a+1,b+1),(a,b+1)}^{f}} & \mf{W}_{a+1,b}^{f}\ar{d}{\mf{w}_{(a+1,b),(a,b)}^{f}} \ar{r}{\mf{H}_{a+1,b}^{f-}} & \mf{H}_{b}^{f-}\ar{d}{id}\\
        \mf{W}_{a,b+1}^{f} \ar{r}{\mf{w}_{(a,b+1),(a,b)}^{f}}\ar{d}{\mf{H}_{a,b+1}^{f+}} & \mf{W}_{a,b}^{f} \ar{r}{\mf{H}_{a,b}^{f-}} \ar{d}{\mf{H}_{a,b}^{f+}}& \mf{H}_{b}^{f-} \\
        \mf{H}_{a}^{f+}\ar{r}{id} &  \mf{H}_{a}^{f+}
    \end{tikzcd}
\end{equation*}
\item For any $a\in \mb{Z}$, we have the commutative diagram (up to equivalence) where all the rows and columns are fiber sequences:
   \begin{equation*}
      \begin{tikzcd}[column sep=2cm]
        \mf{W}_{a+1,a+1}^{f}\ar{r}{id}\ar{d}{\mf{w}_{(a+1,a+1),(a,a+1)}} & \mf{W}_{a+1,a+1}^{f} \ar{d}{\mf{w}_{(a+1,a+1),(a,a)}}\\
        \mf{W}_{a,a+1}^{f}\ar{r}{\mf{w}_{(a-1,a),(a,a)}} \ar{d}{h_{a,a+1}^{+}} & \mf{W}_{a,a}^{f} \ar{r}{h_{a,a}^{-}}\ar{d} &  \mf{H}_{a}^{f-}\ar{d} \\
        \mf{H}_{a}^{f+}\ar{r}{f_{*}(\phi_{a,C_{Z/X}})} & f_{*}pr_{C_{Z/X}*}\mc{O}_{\mb{P}_{Z}(C_{Z/X})}(a) \ar{r}{f_{*}(\psi_{a,C_{Z/X}})} & \mf{H}_{a}^{f-}.
      \end{tikzcd}
    \end{equation*}
  \item 
    For any $b\geq max\{-r+1,0\}$, the canonical map
    $$\mc{O}_{X}\to pr_{f*}(\mc{O}_{\mb{B}l_{Z/X}})$$
    is represented by $\mf{w}_{(0,b),(0,0)}$, which is an equivalence if $r>0$. 
\item The morphism
    $$\mc{O}_{X}\to f_{*}\mc{O}_{Z}$$
    is represented by $h_{0,b}^{f+}$ for any $b\geq max\{-r+1,0\}$.
\end{enumerate}
\end{theorem}

\subsection{A $K$-theoretic comparison theorem  and virtual localizations}
\label{sec:1.3}
 When $S=\mathrm{Spec}(\mb{C})$, the quasi-coherent sheaves $\mf{W}_{a,b}^{f}$ are moreover coherent over $X$. If $r\leq 0$, we take $D_{l}:=\mf{W}_{0,l}$ for $0\leq l\leq -r+1$. Then \cref{thm:nmain} verifies Conjecture 1.10 of the author \cite{yuzhaoderived}:
\begin{theorem}
  \label{conj:1.4}
 We have
  \begin{enumerate}
  \item $D_{-r+1}\cong \mc{O}_{X}$ and $D_{0}\cong pr_{f*}\mc{O}_{\mb{B}l_{Z/X}}$;
  \item
    We have fiber sequences:
    $$D_{l+1}\to D_{l}\to f_{*}det(C_{Z/X})^{-1}S^{-l-r}_{Z}(C_{Z/X})^{\vee}[1-r]$$
  \end{enumerate}
  Moreover, we have the following equality in the Grotheneieck group $G_{0}(X)$:
\begin{equation}
  \label{eq:noref}
  pr_{f*}([\mc{O}_{\mb{B}l_{Z/X}}])=[\mc{O}_{X}]+(-1)^{1-r}f_{*}(\sum_{l=0}^{-r}det(C_{Z/X})^{-1}S^{l}_{Z}(C_{Z/X})^{\vee}).
\end{equation}
\end{theorem}

By  Ciocan-Fontanine and Kapranov \cite{ciocan-fontanine07:virtual}, $[\mc{O}_{\mb{B}l_{Z/X}}]$ and $[\mc{O}_{X}]$ are the virtual $K$-theoretic virtual fundamental classes of $\pi_{0}(\mc{O}_{\mb{B}l_{Z/X}})$ and $\pi_{0}(X)$ respectively, and \cref{eq:noref} is a $K$-theorectic comparison formula for the virtual fundamental classes of derived blow-ups.

Now we moreover assume that $X$ is a derived scheme that is quasi-compact and locally of finite presentation, with an action by a connected reductive algebraic group $G$. Then we have
\begin{theorem}[Hekking-Rydh-Savvas, Theorem 5.3 of \cite{hekking2022stabilizer}]
The derived blow-up $\mb{B}l_{X^{G}}X$ is the derived enhancement of the intrinsic blow-up of fixed locus $\pi_{0}(X^{G})$ in $\pi_{0}(X)$ in the sense of Kiem-Li-Savvas \cite{kiem2017generalized}.
\end{theorem}
In \cref{sec:C} we show that for perfect obstruction theories induced from quasi-smooth derived schemes, the virtual localization formula is a particular case of \cref{eq:noref}.

\begin{remark}
  A more general series of localization theorems (for Chow groups) was developed by Aranha-Khan-Latyntsev-Park-Ravi \cite{aranha2022localization}, which moreover holds for derived Artin stacks. Their and our proofs both rely on the deformation to the normal cone and thus it is interesting to consider the potential relationship between these two proofs.
\end{remark}

\subsection{Desingulization of quasi-smooth derived schemes} The following example explains that the empty scheme is essential in the desingularization theory of quasi-smooth derived schemes:
\begin{example}
Let 
$$X:=\mb{R}Z_{U}^{V}(0)$$ 
be the derived zero locus of zero co-section of a locally free sheaf $V$
over a smooth variety $U$.
Then for any smooth variety $Z\subset X$, by \cref{0524B} we have
$$\mb{B}l_{Z/X}\cong \mb{R}Z_{\mb{B}l_{Z/U}}^{pr_{Z/U}^{*}V\otimes\mc{O}_{\mb{B}l_{Z/U}}(-1)}(0).$$ 

\end{example}
Now we formulate our desingularization theorem:
\begin{theorem}[Desingulization Theorem, \cref{thm:main2}]
\label{thm:1.9}
  Let $X\subset S$ be a quasi-smooth closed embedding (i.e. regular embedding) into a smooth variety $S$. Then there exists a collection of closed embeddings of derived schemes:
  $$f_{i}:Z_{i}\to X_{i}\subset S_{i},
  \quad 0\leq i<n$$
where $n$ is a non-negative integer, such that
  \begin{enumerate}
  \item $X_{0}\cong X$, $S_{0}\cong S$, $X_{n}\cong \emptyset$, $X_{i}\cong \mb{B}l_{Z_{i-1}/X_{i-1}}$ and $S_{i}\cong \mb{B}l_{Z_{i-1}/S_{i-1}}$ for $1\leq i\leq n$;
  \item all $X_{i}\subset S_{i}$ are quasi-smooth closed embeddings into smooth varieties.
    \end{enumerate}
\end{theorem}
Combining with \cref{eq:noref}, we induce the following theorem on the virtual fundamental class of a quasi-smooth derived scheme:
\begin{theorem}[Approximation Theorem, \cref{conj:approx}]
  
  Let $X$ be a quasi-smooth derived scheme such that 
  \begin{itemize}
      \item the rank of $L_{X}$ is constant, which we denote as $vdim(X)$;
      \item $X$ has a closed embedding into a smooth ambient variety.
  \end{itemize}
Then there exists a collection $\{(Z_{i},p_{i},\mc{F}_{i},r_{i})\}|_{1\leq i\leq n}$ such that for any $1\leq i\leq n$,
  \begin{enumerate}
  \item $r_{i}$ is non-positive integer;
  \item $Z_{i}$ is a smooth variety such that $dim(Z_{i})=-r_{i}+vdim(X)$;
  \item  $p_{i}:Z_{i}\to X$ is a proper morphism;
  \item  $\mc{F}_{i}$ is a perfect complex over $Z_{i}$ with tor-amplitude $[0,1]$ and $rank(\mc{F}_{i})=r_{i}$
  \item the following formula holds:
    \begin{equation}
      \label{K-theoretic}
    [\mc{O}_{X}]=\sum_{i=1}^{n}(-1)^{r_{i}}p_{i*}(det(\mc{F}_{i})^{-1}\sum_{j=0}^{-r_{i}}S^{j}(\mc{F}_{i}^{\vee})).
  \end{equation}
  \end{enumerate}
\end{theorem}
\begin{remark}The assumption that $X$ embeds into an ambient smooth variety are probably not needed and we will try to remove it in the future.
\end{remark}
\subsection{Resolutions of the diagonal for derived blow-up of regular embeddings}
Given a locally free sheaf $\mc{F}$ over a derived stack $X$, Beilinson constructed a resolution of the diagonal structure sheaf $\mc{O}_{\mb{P}_{X}({\mc{F}})}$ in 
$$\mb{P}_{X}(\mc{F})\times_{X}\mb{P}_{X}(\mc{F}).$$
is constructed by Beilinson \cite{beilinson1978coherent}, which induces a semi-orthogonal decomposition of the derived category of $\mb{P}_{X}(\mc{F})$ (Beilinson \cite{beilinson1978coherent} for $\mb{P}^{n}$, Orlov \cite{MR1208153} for the relative case and Khan \cite{MR4149835} for the derived algebraic geometry setting).

In \cref{sec:A}, we constructed a resolution of $\mc{O}_{\mb{B}l_{Z/X}}$ in the diagonal embedding of $\mb{B}l_{Z/X}\times_{X}\mb{B}l_{Z/X}$, where 
 $$f:Z\to X$$
 is a quasi-smooth closed embedding of derived stacks such that $L_{Z/X}$ has rank $r$. Like Beilinson's resolution, it also induces a new constructive proof of the semi-orthogonal decomposition theorem of $\mathrm{QCoh}(\mb{B}l_{Z/X})$ (resp. $\mathrm{Perf}(\mb{B}l_{Z/X})$) by Orlov \cite{MR1208153} and Khan (for the derived setting) \cite{MR4149835}. Our new observation is that
\begin{theorem}[\cref{thm:resolution}]
  \label{thm:1.7}
  The diagonal embedding of
  $$\Delta_{pr_{Z/X}}:\mb{B}l_{Z/X}\to \mb{B}l_{Z/X}\times_{X}\mb{B}l_{Z/X}$$ is the projection map of the derived blow-up of $\mb{B}l_{Z/X}\times_{X}\mb{B}l_{Z/X}$ along the closed subscheme $\mb{P}_{Z}(C_{Z/X})\times_{Z}\mb{P}_{Z}(C_{Z/X})$, where $C_{Z/X}$ is a rank $r$ locally free sheaf over $Z$.
\end{theorem}
Combining \cref{thm:1.7} and \cref{thm:nmain}, we have the resolution of diagonals:
\begin{theorem}[Beilinson-type resolution, \cref{thm:fil}]
    We have perfect complexes $\mf{W}_{i,r-1}$ over $\mb{B}l_{Z/X}\times \mb{B}l_{Z/X}$ for $0\leq i\leq r-1$
  $$\mf{W}_{0,r-1}\cong \mc{O}_{\mb{B}l_{Z/X}\times_{X}\mb{B}l_{Z/X}},\quad \mf{W}_{r-1,r-1}\cong \Delta_{\mb{B}l_{Z/X}*}(\mc{O}_{\mb{B}l_{Z/X}}(r-1))$$
  with fiber sequences:
  $$\mf{W}_{i+1,r-1}\to \mf{W}_{i,r-1}\to \bar\beta_{Z/X*} S^{i}(C_{\beta_{Z/X}})$$
  for $0\leq i\leq r-2$. Here $\bar\beta_{Z/X}$ is the closed embedding of $\mb{P}_{Z}(C_{Z/X})\times_{Z}\mb{P}_{Z}(C_{Z/X})$ into $\mb{B}l_{Z/X}\times\mb{B}l_{Z/X}$, and $C_{\beta_{Z/X}}$ is the two-term complex
  $$L_{\mb{P}_{Z}(C_{Z/X}^{1})/Z}\otimes \mc{O}_{\mb{P}_{Z}(C_{Z/X}^{1})}(1)\to \mc{O}_{\mb{P}_{Z}(C_{Z/X}^{2})}(1),$$
  where $C_{Z/X}^{1}$ and $C_{Z/X}^{2}$ are two copies of $C_{Z/X}$, and we abuse the notation to denote $C_{Z/X},\mc{O}_{\mb{P}_{Z}(C_{Z/X}^{1})}(1)$ and $\mc{O}_{\mb{P}_{Z}(C_{Z/X}^{2})}(1)$ as the pull-back of $C_{Z/X},\mc{O}_{\mb{P}_{Z}(C_{Z/X}^{1})}(1)$ and $\mc{O}_{\mb{P}_{Z}(C_{Z/X}^{2})}(1)$ from $Z,\mb{P}_{Z}(C_{Z/X}^{1})$ and $\mb{P}_{Z}(C_{Z/X}^{2})$ to $\mb{P}_{Z}(C_{Z/X}^{1})\times_{Z}\mb{P}_{Z}(C_{Z/X}^{2})$ respectively.
\end{theorem}
\subsection{Quantum loop groups and derived birational geometry of nested quiver varieties}
\label{sec:1.8}
Let $Q=(I,E)$ be a quiver. Given $k\in I$ and $\mbf{v}^{0},\mbf{v}^{1},\mbf{v}^{1'},\mbf{v}^{2},\mbf{w}\in \mb{Z}_{\geq 0}^{I}$ such that $\mbf{v}^{1}=\mbf{v}^{1'}$ and  $\mbf{v}^{2}-\mbf{v}^{1}=\mbf{v}^{1'}-\mbf{v}^{0}=(\delta_{kj})_{j\in I}$ ($\delta$ is the Kronecker symbol), we consider the following triple moduli spaces
\begin{align*}
  \mf{Z}_{k}^{-}(\mbf{v}^{1}):=\mf{P}(\mbf{v}^{0},\mbf{v}^{1})\times_{\mf{M}(\mbf{v}^{0},\mbf{w})}\mf{P}(\mbf{v}^{0},\mbf{v}^{1'})  \\
 \mf{Z}_{k}^{+}(\mbf{v}^{1}):=\mf{P}(\mbf{v}^{1},\mbf{v}^{2})\times_{\mf{M}(\mbf{v}^{2},\mbf{w})}\mf{P}(\mbf{v}^{1'},\mbf{v}^{2}),
\end{align*}
where $\mf{M}(\mbf{v},\mbf{w})$ is the Nakajima quiver variety and $\mf{P}(\mbf{v}_{0},\mbf{v}_{1})$ is the Hecke correspondence defined by Nakajima \cite{10.1215/S0012-7094-94-07613-8,10.1215/S0012-7094-98-09120-7}. In \cite{yuzhaoderived} we proved that
\begin{theorem}[Theorem 1.1 of \cite{yuzhaoderived}]
  \label{thm:1.4}
  Regarding $\mf{P}(\mbf{v}^{0},\mbf{v}^{1})$ (resp. $\mf{P}(\mbf{v}^{1},\mbf{v}^{2})$) as a closed derived subscheme of $\mf{Z}_{k}^{-}(\mbf{v}^{1})$ (resp. $\mf{Z}_{k}^{+}(\mbf{v}^{1})$) through the diagonal embedding, then there is a smooth variety $\mf{Y}_{k}(\mbf{v}^{1})$ such that
  \begin{equation*}
     \mb{B}l_{\mf{P}(\mbf{v}^{0},\mbf{v}^{1})/\mf{Z}_{k}^{-}(\mbf{v}^{1})}\cong \mf{Y}_{k}(\mbf{v}^{1})\cong \mb{B}l_{\mf{P}(\mbf{v}^{1},\mbf{v}^{2})/\mf{Z}_{k}^{+}(\mbf{v}^{1})},    
  \end{equation*}
\end{theorem}
\begin{remark}
  The smooth variety $\mf{Y}_{k}(\mbf{v}^{1})$ is called the Negu\c{t}'s quadruple moduli space, as the first non-trivial case, i.e. the Jordan quiver case, was discovered by Negu\c{t} \cite{neguct2018hecke}.
\end{remark}

\cref{thm:nmain} and \cref{thm:1.4} are related to the weak categorification of the quantum loop (and toroidal) algebra action, due to the following observation: the moduli spaces $\mf{Z}_{k}^{-}(\mbf{v}^{1})$ parametrized stable framed representations $B_{0}\subset B_{1}$ and $B_{0}\subset B_{1}'$. Then $B_{1}/B_{0}$ and $B_{1}'/B_{0}$ induces two line bundles over $\mf{Z}_{k}^{-}(\mbf{v}^{1})$, which we denote as $\mc{L}_{1}$ and $\mc{L}_{1}'$. We can consider similar line bundles  $\mc{L}_{2}$ and $\mc{L}_{2}'$ on $\mf{Z}_{k}^{+}(\mbf{v}^{1})$. The following theorem is a generalization of Proposition 2.39 of Negu\c{t} \cite{neguct2018hecke}:
\begin{proposition}[\cite{zhao2023}]
  The pull backs of $\mc{L}_{1},\mc{L}_{2},\mc{L}_{1}'$ and $\mc{L}_{2}'$ to $\mf{Y}_{k}(\mbf{v}^{1})$ satisfies the relation that
  $$\mc{L}_{1}\mc{L}_{2}'^{-1}\cong \mc{L}_{1}'\mc{L}_{2}^{-1}\cong \mc{O}_{\mf{Y}_{k}(\mbf{v}^{1})}(-\Delta_{\mf{Y}})$$
  where $\mc{O}_{\mf{Y}_{k}(\mbf{v}^{1})}(-\Delta_{\mf{Y}})$ is the exceptional divisor of $\mf{Y}_{k}(\mbf{v}^{1})$ along the derived blow-ups.
\end{proposition}

We recall the fact that the quantum loop algebra has generators $e_{k}^{i},h_{k}^{\pm j},f_{k}^{l}$, where $k\in I, i,j\in \mb{Z},j\in \mb{Z}_{\geq 0}$. Moreover, the $e_{k}^{i}f_{k}^{l}$ action (lifting to the derived category of coherent sheaves), is induced by the Fourier-Mukai kernel $\mc{L}_{1}^{i}\mc{L}_{1}'^{l}$ over $\mf{Z}_{k}^{-}(\mbf{v}^{1})$, and the $f_{k}^{l}e_{k}^{i}$ action is induced by the Fourier-Mukai kernel $\mc{L}_{2}^{l}\mc{L}_{2}'^{i}$ over $\mf{Z}_{k}^{+}(\mbf{v}^{1})$ (strictly speaking, we need to twist a line bundle which depends on the quiver $Q$ and the vectors $\mbf{v}^{1},\mbf{w}$). Thus combining \cref{thm:nmain} and \cref{thm:1.4}, we can explicitly compute the categorical commutator of $e_{k}^{i}f_{k}^{l}$ and $f_{k}^{l}e_{k}^{i}$.

We will explain the details of the above discussions in \cite{zhao2023} and construct a weak categorification of quantum loop and toroidal algebra actions on the Grothendieck group of quiver varieties.

\subsection{Several conjectures}
In this subsection, we propose several related conjectures. 

First, we propose the following conjecture, which suggested the relationship between the (virtual) codimension and the discrepancy, a fundamental concept from the birational geometry:
\begin{conjecture}
  \label{conj:dis}
  We assume all the settings of \cref{thm:main}. Let $L_{\mb{B}l_{Z/X}/S}$ and $L_{X/S}$ be the cotangent complex of $\mb{B}l_{Z/X}$ and $X$ over $S$, then we have
  $$pr_{Z/X}^{*}det(L_{X/S})\cong det(L_{\mb{B}l_{Z/X}/S})\otimes \mc{O}_{\mb{B}l_{Z/X}}(r-1).$$
\end{conjecture}
In Hekking's thesis \cite{Hekking_2022}, Hekking-Khan-Rydh computed the relative cotangent complex of $pr_{Z/X}$. By their computation, we can reduce \cref{conj:dis} to the following conjecture:
\begin{conjecture}
  Let $f:Z\to X$ be a virtual Cartier divisor of derived Artin stacks. Then for any perfect complex $\mc{F}\in \mathrm{Perf(Z)}$, we have
  $$det(f_{*}\mc{F})\cong \mc{O}_{X}(lZ)$$
  where $l$ is the rank of $\mc{F}$.
\end{conjecture}

For the case that $S=\mathrm{Spec}(\mb{C})$ and $X$ is a derived scheme, \cref{conj:dis} was verified by Kollar-Mori \cite[Lemma 2.29]{Kollar-Mori} if both $Z$ and $X$ are smooth, and by us \cite[Theorem 1.13]{yuzhaoderived} if $\mb{B}l_{X/Z}$ is smooth.

Halpern-Leistner \cite{halpern2020derived} proposed the following conjecture between the determinant of the cotangent complex and the dualizing complex:
\begin{conjecture}
  \label{conj:1.14}
  Let $f:X\to Y$ be a morphism between quasi-smooth locally almost of finite presentation derived Artin stacks. Then there is a canonical isomorphism
  $$\omega_{X/Y}\cong det(L_{X/Y})[rank(L_{X/Y})],$$
  where $\omega_{X/Y}$ is the relative dualizing complex. 
\end{conjecture}
\begin{remark}
  In \cite[Proposition 3.1.5]{halpern2020derived}, Halpern-Leistner proved a weak version of \cref{conj:1.14}, which stated that both two sides are isomorphic after restricting to $\pi_{0}(X)$. Moreover, Halpern-Leisner claimed that Dima Arinkin already had sketchy proof through the deformation to the normal cone, which we would be very interested in. 
\end{remark}

By combining \cref{conj:1.14} and \cref{conj:dis}, we propose a more prescise description about $\mf{W}_{1-r,1-r}^{f}$:
\begin{conjecture}
  We have
  $$pr_{f!}\mc{O}_{\mb{B}l_{Z/X}}\cong \mf{W}_{1-r,1-r}^{f}.$$
  When $r\leq 0$, the canonical map
  $$pr_{f!}\mc{O}_{\mb{B}l_{Z/X}}\to \mc{O}_{X}$$
  is represented by $\mf{w}_{(r-1,r-1),(0,r-1)}^{f}$.
\end{conjecture}

Finally, we propose a conjectural relation between the derived blow-ups and the $L_{\infty}$-algebroid structure. Assuming all the setting of \cref{thm:main}, we denote the structure ring $\mc{O}_{Z_{X}^{\infty}}$ as the inverse limit
$$\varprojlim_{n\to \infty}cofib(R_{Z/X}^{n}\to R_{Z/X}^{0}).$$
Inspired by Calaque-C\u{a}ld\u{a}raru-Tu \cite{MR3228441}, we proposed the following conjecture:
\begin{conjecture}
  There exists a minimal $L_{\infty}$-algebroid structure on $C_{Z/X}^{\vee}$ whose Chevalley-Eilenberg algebra is equivalent to $\mc{O}_{Z_{X}^{\infty}}$.
\end{conjecture}
The case that both $X$ and $Z$ are smooth varieties over $\mathrm{Spec}(\mb{C})$ was verified in \cite[Proposition 1.4]{MR3228441}.

\subsection{Organization of the paper}
We introduce the derived algebraic geometry, derived projectivization theory, and derived blow-up theory from \cref{sec:2} to \cref{sec:4}. The generalized vanishing theorems are proved in \cref{sec:5}. In the appendix, we introduce three applications, and the weak categorification of quantum loop and toroidal algebras will be written in a separate paper \cite{zhao2023}. 
\subsection{Acknowledgements}
The paper is written in memory of our best friend Kun Wang, as his insistence on studying mathematics despite serious illness is always a great inspiration to all of us.

Jeroen Hekking shared his thesis with the author, with many details about him and his collaborator's study. The first version of this paper was greatly motivated by private communication with Siqi He, for the construction of a single Kuranishi chart. The relation with the intrinsic blow-up theory was suggested by Zijun Zhou. Part of this paper was first presented in a talk at HKUST, where Wei-Ping Li and Huai-Liang Chang asked many stimulating questions. Qingyuan Jiang and Will Donovan gave many useful comments for the draft of this paper. The author would thank them for their help.

The author is supported by World Premier International Research Center Initiative (WPI initiative), MEXT, Japan, and Grant-in-Aid for Scientific Research grant  (No. 22K13889) from JSPS Kakenhi, Japan.

\section{Notations and Derived Algebraic Geometry}
\label{sec:2}
In this section, we review the language of derived algebraic geometry and introduce the notations we need in our paper. Other than the standard literature like \cite{toen2008homotopical,gaitsgory2019study,SAG}, we also refer to Hekking's thesis \cite{Hekking_2022}, Jiang's monograph \cite{jiang2022derived} and Khan's paper \cite{MR4149835} and adapt their notations.

\subsection{$\infty$-category}
We denote 
\begin{itemize}
    \item $\mathrm{Spc}$ as the $\infty$-category of spaces;
    \item $\mathrm{Mod}$ as the $\infty$-category associated to the unbounded derived category of $\mb{C}$-vector spaces;
    \item $\mathrm{Alg}$ as the $\infty$-category of connective $\mb{E}_{\infty}$-algebras in $\mathrm{Mod}$.
\end{itemize}
As we work over an algebraically closed field of characteristic $0$, $\mathrm{Alg}$ is equivalent to the $\infty$-category of commutative differential graded algebras (or commutative simplicial rings) over $\mb{C}$ concentrated in homological degree (resp. homotopy degree) $\geq 0$.

\subsection{Higher algebra}

An object $R$ of $\mathrm{Alg}$ is called an algebra. We denote $\mathrm{Alg}_{R}$ as the $\infty$-category of $R$-algebras, and $\mathrm{Mod}_{R}$ as the $\infty$-category of $R$-modules.  For $R\in \mathrm{Alg}$ and $M\in \mathrm{Mod}_{R}$, the homotopy groups $\pi_{n}(M)$ are canonically $\pi_{0}(R)$-modules, for all $n\in \mb{Z}$. In particular, $\pi_{n}(R)$ is a $\pi_{0}(R)$-module for all $n\geq 0$.

For $n\geq 0$, we adjoin a free variable in homological degree $n$ to $R$ by writing $R[u]=R[S^{n}]$, which has the universal property that the space of $R$-algebra maps $R[u]\to B$ is equivalent to the space of maps $(S^{n},*)\to (B,0)$ of pointed spaces.

Given a sequence $(\sigma_{1},\cdots,\sigma_{k})$ of elements $\sigma_{i}\in \pi_{n_{i}}(R)$, we define $R/(\sigma_{1},\cdots, \sigma_{k})$ as the pushout diagram
\begin{equation*}
  \begin{tikzcd}
    R[u_{1},\cdots,u_{k}]\ar{r}{z} \ar{d} & R\ar{d} \\
    R\ar{r} & R/(\sigma_{1},\cdots,\sigma_{k})
  \end{tikzcd}
\end{equation*}
where the $u_{i}$ are free in homogoical degree $n_{i}$ and the map $S$ is induced by lifts $S^{n_{i}}\to R$ of $\sigma_{i}$ and $z$ sends each $u_{i}$ to $0$.
\subsection{Derived schemes and stacks}
We denote
\begin{itemize}
    \item $\mathrm{Aff}:=\mathrm{Alg}^{op}$ as the $\infty$-category of affine derived schemes over $\mathrm{Spec}(\mb{C})$;
    \item a derived stack as a functor $X:\mathrm{Aff}^{op}\to \mathrm{Spc}$ which satisfies \'{e}tale descent conditions;
    \item a derived scheme as a derived stack which has a cover of open immersions by affine derived schemes;
    \item $\mathrm{Stk}$ (resp. $\mathrm{Sch}$) as the $\infty$-category of derived stacks (resp. derived schemes). 
\end{itemize}

The definition of derived Artin stacks is defined in \cite[Vol. \RN{1}, Section 4.1]{gaitsgory2019study}, which is more involved, and thus we recommend the readers read the original literature for the precise definition. The main property we will use is that for a derived Artin stack, there exists a smooth cover by a derived scheme so we can consider the fppf descent.

Any derived stack $X$ admits an underlying classical stack, which we denote as $\pi_{0}(X)$. Moreover, if $X$ is a derived scheme, Deligne-Mumford stack or Artin stack, then $\pi_{0}(X)$ is a classical such.

\begin{remark}
  Hekking's thesis \cite{Hekking_2022} considered a more generalized definition than a derived Artin stack, which they called the derived algebraic stack. Instead, they called a derived Artin stack a derived geometric stack. All the results in our paper should also work for derived algebraic stacks.
\end{remark}

\subsection{Quasi-smooth morphisms}
\begin{definition}
  A map of derived Artin stacks $f:Z\to X$ is called quasi-smooth if $f$ has a relative cotangent complex $L_{Z/X}$ which is perfect and has $Tor$-amplitude $[-1,1]$.
\end{definition}

\begin{lemma}[Proposition 2.3.8 and Proposition 2.3.14 of \cite{khan2018virtual}]
\label{131666}
  Let $f:Z\to X$ be a map of derived schemes. Then
  \begin{itemize}
  \item $f$ is a quasi-smooth closed embedding if and only if $f$ is locally of the form $\mathrm{Spec}(A/(f_{1},\cdots,f_{n}))\to \mathrm{Spec}(A)$ for some $f_{1},\cdots,f_{n}\in \pi_{0}(A)$.
  \item $f$ is a quasi-smooth map if and only if it admits, Zarisky-locally on $X$, a factorization
    $$Z\xrightarrow{g}Y\xrightarrow{g'} X$$
    where $g$ is a quasi-smooth closed embedding and $g'$ is smooth.
    \end{itemize}
\end{lemma}

\subsection{Quasi-coherent sheaves}
Given a derived stack $X$, the stable $\infty$-category of quasi-coherent sheaves $\mathrm{QCoh}(X)$ is the limit
\begin{equation*}
  \mathrm{QCoh}(X)=\varprojlim_{\mathrm{Spec}(A)\to X} \mathrm{Mod}_{A}
\end{equation*}
taken over all morphisms from derived affine schemes to $X$.

We denote $\mathrm{Coh}(X)$ as the full subcategory of $\mathrm{QCoh}(X)$ which contains bounded complexes whose cohomology sheaves are coherent $\mc{O}_{\pi_{0}(X)}$-modules. We denote $\mathrm{Perf}(X)$ as the full subcategory of $\mathrm{QCoh}(X)$ which contains perfect objects.

Given a quasi-coherent sheaf $\mc{E}$ on $X$, we say that $\mc{E}$ is connective if $\pi_{i}(\mc{E})\cong 0$ for all $i<0$.

\subsection{Derived functors}
Given a derived $X$, \cite[\S 6.2.6]{SAG} introduces a symmetric monoidal structure on $\mathrm{QCoh}(X)$, where the structure sheaf $\mc{O}_{X}$ is the unit element.

Given a map of derived stacks $h:X\to Y$, the pull-back of modules induces a pull-back functor:
\begin{equation*}
  h^{*}:\mathrm{QCoh}(Y)\to \mathrm{QCoh}(X)
\end{equation*}
By Corollary 1.3.11 of \cite{gaitsgory2019study}, the functor $\mathrm{QCoh}$ satisfies descent in flat topology (under the pull-back functors).

The following theorem collected the properties of push-forward functors for quasi-coherent sheaves. Particularly, it shows that regarding $\mathrm{QCoh}$ as a functor under the push forward, it is stable with the pull-backs:
\begin{theorem}[Lipman-Neeman, Lurie, Theorem 3.7 of \cite{jiang2022derived}]
  \label{thm:lipman}
  Let $h:X\to Y$ be a map of derived stacks that is quasi-compact, quasi-separated, and schematic. Let $h^{*}:\mathrm{QCoh}(Y)\to \mathrm{QCoh}(X)$ denote the pullback functor. Then
  \begin{enumerate}
  \item The pullback $h^{*}$ admits a right adjoint $h_{*}:\mathrm{QCoh}(X)\to \mathrm{QCoh}(Y)$, called pushforward functor, which preserves small colimits.
  \item For every pullback diagram of derived stacks:
    \begin{equation*}
      \begin{tikzcd}
        X'\ar{r}{g'}\ar{d}{h'} & X\ar{d}{h} \\
        Y'\ar{r}{g} & Y,
      \end{tikzcd}
    \end{equation*}
    the canonical base change transformation, i.e. the Beck-Chevalley transformation $g^{*}h_{*}\to h_{*}'g'^{*}$ is an equivalence of functors from $\mathrm{QCoh}(X)$ to $\mathrm{QCoh}(Y')$.
  \item For every pair of objects $\mc{F}\in \mathrm{QCoh}(X)$ and $\mc{G}\in \mathrm{QCoh}(Y)$ the canonical map $\mc{F}\otimes f_{*}\mc{G}\to f_{*}(f^{*}\mc{F}\otimes \mc{G})$ is an equivalence.
  \item  If $h$ is proper, locally almost of finite presentation and local of finite Tor-amplitude, then the push forward functor $h_{*}:\mathrm{QCoh}(X)\to \mathrm{QCoh}(Y)$ preserves perfect objects.
  \end{enumerate}
\end{theorem}
\begin{remark}
  Theorem 3.7 of \cite{jiang2022derived} actually holds even for prestacks.  The definition of quasi-compact and quasi-separated property is like the classical algebraic geometry situation, which we will not introduce in our paper but refer to \cite{SAG}. Theorem 3.7 of \cite{jiang2022derived} contained more information about other adjoint functors $f^{!}$ and $f_{!}$, which we will study in future work.
\end{remark}

\subsection{Symmetric powers of quasi-coherent sheaves}
Given a quasi-coherent sheaf $\mc{F}$ over a derived stack $X$ and an integer $n\geq 0$, we denote the quasi-coherent sheaf
$$S^{n}_{X}(\mc{F})$$
as the nth symmetric power of $\mc{F}$, following the notation of Section 2 of \cite{jiang2022derived} or Section 25.2 of \cite{SAG} (except replacing the notation $Sym^{n}$ by $S^{n}$). We denote
$$S^{\bullet}_{X}(\mc{F}):=\oplus_{n=0}^{\infty}S^{n}_{X}(\mc{F})$$
as a quasi-coherent $\mc{O}_{X}$-algebra and denote $S_{X}^{n}(\mc{F}):= 0$ when $n<0$. We will omit $X$ in $S^{n}_{X}(\mc{F})$ if $X$ is clear from the context. We denote $\mc{F}^{\vee}$ as the derived dual of $\mc{F}$.

\begin{example}
    Let $\mc{F}\cong \{V\to W\}$ where $V$ and $W$ are locally free sheaves over a smooth variety $X$. Then we have
    \begin{align*}
    S^n_X(\mc{F})\cong\{\wedge^n_X V\to \cdots\to S_X^n W\} \\
    S^n_X(\mc{F}^\vee)\cong\{S^n_X W^\vee\to \cdots \to \wedge_X^n V^\vee \}[-n].
    \end{align*}
The heuristic argument can be generalized to the case that $X$ is a derived stack and $V,W$ are quasi-coherent sheaves over $X$, by Lemma 4.16 of \cite{jiang2022derived}.
\end{example}
\section{Jiang's derived projectivization theory and the generalized Serre theorem}
\label{sec:3}
In this section, we assume that $X$ is a derived stack, with connective quasi-coherent sheaves $\mc{F}$ and $\mc{F}'$. Jiang \cite{jiang2022derived} defined the derived projectivization $\mb{P}_{X}(\mc{F})$ over $X$, and proved that the projection
$$pr_{F}:\mb{P}_{X}(\mc{F})\to X$$
is schematic. Moreover, when $\mc{F}$ is of perfect amplitude contained in $[0,1]$ and has constant rank $r$, Jiang computed the push-forward of tautological line bundles over $\mb{P}_{X}(\mc{F})$, which generalized the theorem of Serre and Lurie. In this section, we introduce Jiang's generalized Serre theorem.

\subsection{Derived Affine Cones}

\begin{definition}[Derived affine cone in Section 4.1 of \cite{jiang2022derived}, or quasi-coherent bundle in Definition 2.4 of Paper B of \cite{Hekking_2022}]
 The affine cone of $\mc{F}$, which we denoted as $\mb{V}_{X}(\mc{F})$, is defined as the derived stack
  $$\mb{V}_{X}(\mc{F})(U)=Map(f^{*}\mc{F},\mc{O}_{U})$$
 \end{definition}
\begin{proposition}[Proposition 4.3 of \cite{jiang2022derived}, or Proposition 2.11 of Paper B of \cite{Hekking_2022}]
  The derived stack $\mb{V}_{X}(\mc{F})$ is representable by the relative affine derived scheme over $X$
  $$\Pi_{F}:\mb{V}_{X}(\mc{F})=\mathrm{Spec}(S_{X}^{*}(\mc{F}))\to X.$$
  The universal morphism of $\mc{O}_{\mb{V}_{X}(\mc{F})}$-module $\varrho_{F}:\Pi_{F}^{*}F\to \mc{O}_{\mb{V}_{X}(F)}$ is induced from the canonical morphism $F\to S^{*}_{X}(F)$ by adjunction.
\end{proposition}
\begin{definition}[Universal and hyperplane $\mb{G}_{m}$-action]
   For any connective quasi-coherent sheaf $\mc{F}$ over $X$, we define the universal (resp. hyperplane) $\mb{G}_{m}$-action on $\mb{V}_{X}(\mc{F})$ such that $S^{k}_{X}\mc{F}$ has grading $k$ (resp. $-k$).

  Paticularly, We define the hyperplane $\mb{G}_{m}$-action on $\mb{A}^{1}$ by the grading that $\mb{A}^{1}=\mathrm{Spec}\ \mb{C}[t^{-1}]$, where $t^{-1}$ has grade $-1$.
\end{definition}

\begin{definition}[Derived zero locus of a co-section]
  By Section 4.1 of \cite{jiang2022derived}, a morphism $\gamma:\mc{F}\to \mc{F}'$ over $X$ induces a map, which we denote as
  $$\mb{V}_{X}(\gamma):\mb{V}_{X}(\mc{F}')\to \mb{V}_{X}(\mc{F}).$$

  A co-section $\tau:\mc{F}\to \mc{O}_{X}$ induces a $X$-scheme closed embedding, which we denote as
  $$i_{\tau}:X\to \mb{V}_{X}(\mc{F}).$$
  Particularly, we denote $i_{0}$ or $i_{0,\mc{F}}$
  $$i_{0,\mc{F}}:X\to \mb{V}_{X}(\mc{F})$$
  as the closed embedding induced from the zero cosection $0:\mc{F}\to \mc{O}_{X}$. We denote
  $$\mb{V}_{X}^{*}(\mc{F}):=\mb{V}_{X}(\mc{F})-X.$$
  
  We define the derived zero locus of the cosection $\tau:\mc{F}\to \mc{O}_{X}$, which is denoted as $\mb{R}Z_{X}^{\mc{F}}(\tau)$, through the following Cartesian diagram:
  \begin{equation*}
    \begin{tikzcd}
      \mb{R}Z_{X}^{\mc{F}}(\tau)\ar{r}\ar{d} & X\ar{d}{i_{\tau}}\\
      X \ar{r}{i_{0}} & \mb{V}_{X}(\mc{F}).
    \end{tikzcd}
  \end{equation*}
\end{definition}
\subsection{Derived projectivization of complexes}
\begin{definition}[Section 1.1 and 4.2 of \cite{jiang2022derived}]
   The derived projectivization $\mb{P}_{X}(\mc{F})$ of $\mc{F}$ over $X$ is defined as the derived stack that for every map $\eta:T\to X$ from a derived scheme $T$, the space of $T$-points $\mb{P}_{X}(\mc{F})(\eta)$ is the space of morphisms of quasi-coherent sheaves $\eta^{*}\mc{F}\to \mc{L}$ which induce surjective morphisms $\pi_{0}(\eta^{*}\mc{F})\to \pi_{0}(\mc{L})$, where $\mc{L}$ is a line bundle on $T$. 
\end{definition}

\begin{proposition}[Proposition 4.18 of \cite{jiang2022derived}]
  The functor $\mb{P}_{X}(\mc{F})$ is represented by a relatively derived scheme over $X$.
\end{proposition}

\begin{definition}
We denote
\begin{itemize}
    \item $\mc{O}_{\mb{P}_{X}(\mc{F})}(1)$ as the universal line bundle on $\mb{P}_{X}(\mc{F})$;
    \item $pr_{\mc{F}}:\mb{P}_{X}(\mc{F})\to X$ as the canonical projection map;
    \item $\rho_{\mc{F}}:pr_{\mc{F}}^{*}\mc{F}\to \mc{O}_{\mb{P}_{X}(\mc{F})}(1)$
    as the tautological quotient morphism.
\end{itemize}

\end{definition}
\begin{theorem}[Euler fiber sequence, Theorem 4.27 of \cite{jiang2022derived}]
  The cotangent complex $L_{\mb{P}_{X}(\mc{F})/X}$ is the fiber of $\rho_{\mc{F}}\otimes \mc{O}_{\mb{P}_{X}(\mc{F})}(-1)$.
\end{theorem}
\begin{example}
  \label{ex:215}
  We consider the projection morphism:
  $$pr_{\mc{F}}\circ\Pi_{\mc{O}_{\mb{P}_{X}(\mc{F})}(1)}:\mb{V}_{\mb{P}_{X}(\mc{F})}(\mc{O}_{\mb{P}_{X}(\mc{F})}(1))\to X.$$
  Over $\mb{V}_{\mb{P}_{X}(\mc{F})}(\mc{O}_{\mb{P}_{X}(\mc{F})}(1))$, the composition
  \begin{align*}
    (pr_{\mc{F}}\circ \Pi_{\mc{O}_{\mb{P}_{X}(\mc{F})}(1)})^{*}\mc{F}&\xrightarrow{\Pi_{\mc{O}_{\mb{P}_{X}(\mc{F})}(1)}^{*}(\rho_{\mc{F}})}\Pi_{\mc{O}_{\mb{P}_{X}(\mc{F})}(1)}^{*}(\mc{O}_{\mb{P}_{X}(\mc{F})}(1))\\
    & \xrightarrow{\varrho_{\mc{O}_{\mb{P}_{X}(\mc{F})(1)}}} \mc{O}_{\mb{V}_{\mb{P}_{X}(\mc{F})}(\mc{O}_{\mb{P}_{X}(\mc{F})}(1))}  \end{align*}
induces a map
$$\epsilon_{\mc{F}}:\mb{V}_{\mb{P}_{X}(\mc{F})}(\mc{O}_{\mb{P}_{X}(\mc{F})}(1))\to \mb{V}_{X}(\mc{F})$$
which forms the commutative diagram:
  \begin{equation}
    \label{cd:projection}
    \begin{tikzcd}
      \mb{P}_{X}(\mc{F})\ar{r}{i_{0}'} \ar{d}{pr_{\mc{F}}} &  \mb{V}_{\mb{P}_{X}(\mc{F})}(\mc{O}_{\mb{P}_{X}(\mc{F})}(1))\ar{d}{\epsilon_{\mc{F}}} \\
      X \ar{r}{i_{0}} & \mb{V}_{X}(\mc{F})
    \end{tikzcd}
  \end{equation}
  where $i_{0}$ and $i_{0}'$ are the zero cosections.
\end{example}

\begin{lemma}
  The diagram \cref{cd:projection} induces an isomorphism
  $$\mb{V}_{X}^{*}(\mc{F})\cong \mb{V}_{\mb{P}_{X}(\mc{F})}^{*}(\mc{O}_{\mb{P}_{X}(\mc{F})}(1))$$
  and moreover, we have
  $$\mb{P}_{X}(\mc{F})\cong [\mb{V}_{X}^{*}(\mc{F})/\mb{G}_{m}]$$
  where the $\mb{G}_{m}$-action is the universal action.
\end{lemma}
\begin{proof}
  We only need to construct an inverse of $\epsilon_{\mc{F}}|_{\mb{V}_{X}^{*}(\mc{F})}$. We notice that the universal co-section
  $$\Pi_{\mc{F}}^{*}\mc{F}\to \mc{O}_{\mb{V}_{X}(\mc{F})}$$
  is surjective on $\pi_{0}(\mb{V}_{X}^{*}(\mc{F}))$.
  It induces a morphism
  $$\alpha_{\mc{F}}:\mb{V}_{X}^{*}(\mc{F})\to \mb{P}_{X}(\mc{F})$$
  and an equivalence
$$\beta_{F}:\alpha_{F}^{*}\mc{O}_{\mb{P}_{X}(\mc{F})}(1)\to\mc{O}_{\mb{V}_{X}(\mc{F})},$$
  which induces a map from $\mb{V}_{X}^{*}(\mc{F})$ to $\mb{V}_{\mb{P}_{X}(\mc{F})}^{*}(\mc{O}_{\mb{P}_{X}(\mc{F})}(1))$ and is an inverse of $\epsilon_{\mc{F}}|_{\mb{V}_{X}^{*}(\mc{F})}$.
\end{proof}

\subsection{An example of derived affine cones and projectivization}
The following example is a simplified version of Proposition 4.11, Corollary 4.32, and Proposition 4.33 of \cite{jiang2022derived}, which described the derived affine cones and projectivizations of a two-term complex of locally free sheaves:

\begin{example}
    \label{prop:jiang433}
  Let $\tau:V\to W$ be a morphism of locally free sheaves over $X$. Then we have the Cartesian diagram:
  \begin{equation*}
    \begin{tikzcd}
       \mb{V}_{X}(\tau)\ar{r}\ar{d} & \mb{V}_{X}(W) \ar{d}{\mb{V}_{X}(\tau)} \\
      X \ar{r}{i_{0,V}} & \mb{V}_{X}(V)
    \end{tikzcd}
  \end{equation*}
  Over $\mb{P}_{X}(W)$ we consider the map
  $$taut:pr_{W}^{*}V\otimes \mc{O}_{\mb{P}_{X}(W)}(-1)\to \mc{O}_{\mb{P}_{X}(W)}$$
  as the composition of
  $$pr_{W}^{*}V\xrightarrow{pr_{W}^{*}(\tau)}pr_{W}^{*}W\xrightarrow{\rho_{W}}\mc{O}_{\mb{P}_{X}(W)}(-1)$$
  tensoring by $\mc{O}_{\mb{P}_{X}(W)}(1)$. Then $\mb{P}_{X}(\tau)$ is the derived zero locus of $taut$. Paticularly, if $W\cong \mc{O}_{X}$, then $\mb{P}_{X}(\tau)$ is the derived zero locus of $\tau$ over $X$.
\end{example}

\subsection{Generalized Serre Theorem}
Now we state Jiang's generalized Serre theorem:
\begin{theorem}[Generalized Serre Theorem \cite{jiang2022derived}]
  \label{thm:serre2}
  If $\mc{F}$ is of perfect-amplitude contained in $[0,1]$ and has constant rank $r$, then
  \begin{enumerate}
      \item the projection map $pr_{\mc{F}}:\mb{P}_{X}(\mc{F})\to X$ is proper, quasi-smooth of relative virtual dimension $r-1$;
      \item let $$\phi_{d,\mc{F}}:S^{d}(\mc{F})\to pr_{\mc{F}*}\mc{O}_{\mb{P}_{X}(\mc{F})}(d)$$ 
      be the map induced by the canonical map
  $$pr_{\mc{F}}^{*}S^{d}(\mc{F})\xrightarrow{S^{d}(\rho_{\mc{F}})} \mc{O}_{\mb{P}_{X}(\mc{F})}(d)$$
  and the adjoint relation. Then the co-fiber of of $\phi_{-r,\mc{F}}[r-1]$ is a line bundle on $X$, which we denote as $det(\mc{F})^{\vee}$;
  \item the canonical map
    $$\mc{O}_{\mb{P}_{X}(\mc{F})}(d)\xrightarrow{S^{-d-r}(\rho_{\mc{F}}^{\vee})\otimes\mc{O}_{\mb{P}_{X}(\mc{F})}(-r)} pr_{\mc{F}}^{*}S^{-d-r}(\mc{F}^{\vee})\otimes \mc{O}_{\mb{P}_{X}(\mc{F})}(-r) $$
    induces a canonical map
    $$pr_{\mc{F}*}(\mc{O}_{\mb{P}_{X}(\mc{F})}(d))\xrightarrow{\psi_{d,\mc{F}}}det(\mc{F})^{\vee}S^{-d-r}(\mc{F}^{\vee})[1-r].$$
    We denote $\phi_{\mc{F}}:=\oplus_{d\in \mb{Z}}\phi_{d,\mc{F}}$ and $\psi_{\mc{F}}:=\oplus_{d\in \mb{Z}}\psi_{d,\mc{F}}$. Then we have a canonical fiber sequence of $\mb{Z}$-graded quasi-coherent $S_{X}^{*}\mc{F}$-modules over $X$:
     \begin{equation}
      \label{eq:serre}
   \bigoplus_{d\in \mb{Z}} S^{d}(\mc{F})\xrightarrow{\phi_{\mc{F}}}pr_{\mc{F}*}(\mc{O}_{\mb{P}_{X}(\mc{F})}(d))\xrightarrow{\psi_{\mc{F}}}\bigoplus_{d\in \mb{Z}}(S^{-r-d}\mc{F})^{\vee}\otimes (det\mc{F})^{\vee}[1-r];  
    \end{equation}
    \item $\psi_{d,\mc{F}}$ (resp. $\phi_{d,\mc{F}}$) is a canonical equivalence if $d<0$ (resp. $d\geq -r+1$).
  \end{enumerate}
\end{theorem}
\begin{remark}
  When $\mc{F}$ is equivalent to a two-term complex of locally free sheaves $\{V\to W\}$, Jiang \cite{jiang2022derived} proved that $det(\mc{F})\cong det(W)det(V)^{-1}$, which is compatible with the definition from perfect complexes.
\end{remark}
  
\section{Hekking's thesis}
\label{sec:4}
In this section, we will introduce the derived blow-up theory in Hekking's thesis \cite{Hekking_2022}, particularly Paper B of \cite{Hekking_2022} by Hekking-Khan-Rydh. 

\begin{remark}
  We will introduce their theory in the convenience of introducing all the concepts and theorems we need rather than following the logical order. We strongly recommend that readers read Hekking's thesis, as many important contents are not treated in our introduction.
\end{remark}

In this section, we will always assume that $f:Z\to X$ is a closed embedding of derived stacks over $\mb{C}$, and we also use the pair $(Z,X)$ to denote the closed embedding $f$.
\subsection{Virtual Cartier divisors}
\begin{definition}[Section 3.1.1 and Section 3.2.1 of \cite{khan2018virtual}]
  A virtual Cartier divisor over $X$ is a quasi-smooth
closed embedding $Z\hookrightarrow X$ such that $C_{Z/X}$ is a line bundle over $Z$.

  A generalized Cartier divisor over $X$ is a pair $(\mc{L},s)$, where $\mc{L}$ is a line bundle over $X$ and $s$ is a co-section of $\mc{L}$.
\end{definition}

\begin{lemma}[Proposition 3.2.6 of \cite{khan2018virtual}]
  \label{lem:2018}
  The virtual Cartier divisors and generalized Cartier divisors are both classified by the derived stack 
  $$[\mb{A}^{1}/\mb{G}_{m}].$$
  Moreover, the universal virtual divisor is $0:B\mb{G}_{m}\to [\mb{A}^{1}/\mb{G}_{m}]$.
\end{lemma}

\begin{definition}[Definition 4.1.1 of \cite{khan2018virtual}, or Definition 3.5.2 of \cite{Hekking_2022}]
    A virtual Cartier divisor over the pair $(Z,X)$ is a commutative diagram:
  \begin{equation*}
    \begin{tikzcd}
      D \ar{r}\ar{d}{g} & T \ar{d} \\
      Z \ar{r} & X
    \end{tikzcd}
  \end{equation*}
  such that $D\to T$ is a virtual Cartier divisor. Moreover, we say a virtual Cartier divisor $(T,D)$ over $(Z,X)$ is strict if 
    \begin{itemize}
        \item the underlying square of classical schemes is Cartesian;
        \item the induced map $g^{*}C_{Z/X}\to C_{D/T}$ is surjective on $\pi_{0}$.
    \end{itemize}  
\end{definition}
\begin{remark}
  We adapt the notation of \cite{Hekking_2022}, where the strict virtual Cartier divisor is called ``the virtual Cartier divisor'' in \cite{khan2018virtual}.
\end{remark}

\subsection{Derived blow-ups}

\begin{definition}[Derived blow-ups]
  The derived blow-up of $Z$ in $X$ is the stack over $X$ which classifies strict virtual Cartier divisors over $(Z,X)$, and denote is as $\mb{B}l_{Z/X}$.  The natural projection map is denoted as
  $$pr_{f}\text{ or }pr_{Z/X}:\mb{B}l_{Z}X\to X.$$
  Over $\mb{B}l_{Z/X}$, we have the following universal strict virtual divisor diagram:
  \begin{equation*}
    \begin{tikzcd}
      E_{Z/X}\ar{r}{\iota_{Z/X}}\ar{d} & \mb{B}l_{Z/X}\ar{d}{pr_{Z/X}} \\
      Z\ar{r} & X
    \end{tikzcd}
  \end{equation*}
  We denote $\mc{O}_{\mb{B}l_{Z/X}}(1)$ as the generalized Cartier divisor corresponding to $E_{Z/X}$.
\end{definition}
\begin{example}
    If both $Z$ and $X$ are smooth varieties over $\mathrm{Spec}\ \mb{C}$, then the derived blow-up and the classical blow-up coincide. A more detailed criterion about if the classical and derived blow-ups coincide was studied in \cite{yuzhaoderived}.
\end{example}
\begin{proposition}[Stable under pull-back, Proposition 5.5 of Paper B of \cite{Hekking_2022}]
  \label{prop:func}
  A Cartesian diagram of derived schemes:
  \begin{equation*}
    \begin{tikzcd}
      Z'\ar{r}\ar{d}{f'} & Z' \ar{d}{f}\\
      X'\ar{r}{g} & X,
    \end{tikzcd}
  \end{equation*}
 where $f$ is a closed embedding, induces the following Cartesian diagram:
  \begin{equation*}
    \begin{tikzcd}
      E_{Z'/X'} \ar{r} \ar{d}{\iota_{Z'/X'}} & E_{Z/X} \ar{d}{\iota_{Z/X}} \\
      \mb{B}l_{Z'}X' \ar{r} \ar{d}{pr_{Z'/X'}} & \mb{B}l_{Z}X \ar{d}{pr_{Z/X}} \\
      X' \ar{r}{g} & X.
    \end{tikzcd}
  \end{equation*}
\end{proposition}

\subsection{Deformation to the normal cone}

Given a closed embedding of derived stacks $f:Z\to X$, we consider the hyperplane $\mb{G}_{m}$-action on $X\times \mb{A}^{1}$ and the natural closed embedding
$$X\to X\times \mb{A}^{1},\quad x\to (x,0).$$

The canonical morphism of line bundles:
$$\mc{O}_{\mb{B}l_{Z/X}}\to \mc{O}_{\mb{B}l_{Z/X}}(-1)$$
induces a map from $\mb{V}_{\mb{B}l_{Z/X}}(\mc{O}_{\mb{B}l_{Z/X}}(-1))$ to $\mb{B}l_{Z/X}\times \mb{A}^{1}$, and we consider the composition with the projection to $X\times \mb{A}^{1}$, with  the hyperplane $\mb{G}_{m}$-actions.  Then we have the following strict virtual Cartier divisor
\begin{equation*}
  \begin{tikzcd}
    \mb{V}_{E_{Z/X}}(\mc{O}_{\mb{B}l_{Z/X}}(-1)|_{E_{Z/X}}) \ar{r} \ar{d}& \mb{V}_{\mb{B}l_{Z/X}}(\mc{O}_{\mb{B}l_{Z/X}}(-1))\ar{d}\\
    Z \ar{r} & X\times \mb{A}^{1} 
  \end{tikzcd}
\end{equation*}
which induces a map from $\mb{V}_{\mb{B}l_{Z/X}}(\mc{O}_{\mb{B}l_{Z/X}}(-1))$ to $\mb{B}l_{Z/X\times \mb{A}^{1}}$. By restricting to $\mb{B}l_{Z/X}$ and $$\mb{V}_{\mb{B}l_{Z/X}}^{*}(\mc{O}_{\mb{B}l_{Z/X}}(-1))\cong \mb{V}_{\mb{B}l_{Z/X}}^{*}(\mc{O}_{\mb{B}l_{Z/X}}(1))$$
respectively, it induces two maps
\begin{align}
  \label{eq:051731}
  \mb{B}l_{Z/X}\to \mb{B}l_{Z/X\times \mb{A}^{1}}, \quad \mb{V}_{\mb{B}l_{Z/X}}^{*}(\mc{O}_{\mb{B}l_{Z/X}}(1))\to \mb{B}l_{Z/X\times \mb{A}^{1}},
\end{align}
where the $\mb{G}_{m}$-action on $ \mb{V}_{\mb{B}l_{Z/X}}^{*}(\mc{O}_{\mb{B}l_{Z/X}}(1))$ is universal.  We also notice that $(Z,Z\times \mb{A}^{1})$ is a strict virtual divisor of $(Z,X\times \mb{A}^{1})$, which induces a map
\begin{equation}
  \label{eq:051732}
Z\times \mb{A}^{1}\to \mb{B}l_{Z/X\times \mb{A}^{1}}.  
\end{equation}

\begin{lemma}[Lemma 5.8 of Paper B of  \cite{Hekking_2022}]
  \label{lm:cl}
  The canonical map
  $$\mb{B}l_{Z/X}\to \mb{B}l_{Z/X\times \mb{A}^{1}}$$
  in \cref{eq:051731} is a closed embedding.
\end{lemma}
\begin{definition}
  We define
  $$D_{Z/X}:=\mb{B}l_{Z/X\times \mb{A}^{1}}-\mb{B}l_{Z/X}$$
  as the deformation space of $f$, with the projection map
  $$d_{Z/X}:D_{Z/X}\to X\times \mb{A}^{1}.$$
\end{definition}
\begin{theorem}[Theorem 3.4.1, Theorem 3.5.4 and Theorem 6.3.3 of \cite{Hekking_2022}, and Proposition 4.16 of Paper B of \cite{Hekking_2022}]
  \label{thm:bigdiagram} We have the following properties about the deformation space $D_{Z/X}$:
  \begin{enumerate}
  \item \label{Rees1} The projection map $d_{Z/X}:D_{Z/X}\to X\times \mb{A}^{1}$ is affine, and induces a $\mb{G}_{m}$-equivariant Cartesian diagram
    \begin{equation}
      \label{eq:051733}
    \begin{tikzcd}
      \mb{V}_{Z}(C_{Z/X})\ar{r}\ar{d}{f\circ \Pi_{C_{Z/X}}} & D_{Z/X}\ar{d}{d_{Z/X}} & X\times \mb{G}_{m}\ar{l}\ar{d}{id}\\
      X\times \{0\} \ar{r} & X\times \mb{A}^{1} & X\times \mb{G}_{m}\ar{l},
    \end{tikzcd}
    \end{equation}
   where the $\mb{G}_{m}$-action on $\mb{V}_{Z}(C_{Z/X})$ is universal.
  \item The map $Z\times \mb{A}^{1}\to \mb{B}l_{Z/X\times {A}^{1}}$ in \cref{eq:051732} factors through a closed embedding in $D_{Z/X}$. The map $$\mb{V}_{\mb{B}l_{Z/X}}^{*}(\mc{O}_{\mb{B}l_{Z/X}}(1))\to \mb{B}l_{Z/X\times \mb{A}^{1}}$$
  in \cref{eq:051731} factors through an open embedding into $D_{Z/X}$, which we denote as
  \begin{equation}
  \label{06052132}
      j_{Z/X}:\mb{V}_{\mb{B}l_{Z/X}}^{*}(\mc{O}_{\mb{B}l_{Z/X}}(1))\to D_{Z/X}.
  \end{equation} 
  Under the map $j_{X/Z}$, $\mb{V}_{\mb{B}l_{Z/X}}^{*}(\mc{O}_{\mb{B}l_{Z/X}}(1))$ is identified as $D_{Z/X}-Z\times \mb{A}^{1}. $Moreover, it induces a  $\mb{G}_{m}$-equivariant Cartesian diagram
    \begin{equation}
      \label{eq:051734}
    \begin{tikzcd}
      Z\times \{0\} \ar{r}\ar{d}{i_{0,C_{Z/X}}} & Z\times \mb{A}^{1}\ar{d} & Z\times \mb{G}_{m}\ar{d}\ar{l} \\
      \mb{V}_{Z}(C_{Z/X})\ar{r} & D_{Z/X} & X\times \mb{G}_{m}\ar{l}, \\
    \mb{V}_{E_{Z/X}}^{*}(\mc{O}_{\mb{P}_{Z}(C_{Z/X})}(1)|_{E_{Z/X}})\ar{u} \ar{r} & \mb{V}_{\mb{B}l_{Z/X}}^{*}(\mc{O}_{\mb{B}l_{Z/X}}(1)) \ar{u}{j_{Z/X}}   & (Z-X)\times \mb{G}_{m} \ar{l}\ar{u} \\
    \end{tikzcd}
  \end{equation}
  \end{enumerate}
\end{theorem}
\begin{corollary}
  We have the exceptional divisor
  $$E_{Z}X\cong \mb{P}_{Z}(C_{Z/X})$$
  and the projection map to $Z$ is represented by $pr_{C_{Z/X}}$. Moreover, we have
  $$\mc{O}_{\mb{B}l_{Z/X}}(1)|_{E_{Z}X}\cong \mc{O}_{\mb{P}_{Z}(C_{Z/X})}(1) $$
   and the canonical following fiber sequence of quasi-coherent sheaves over $\mb{B}l_{Z/X}$:
\begin{displaymath}
  \mc{O}_{\mb{B}l_{Z}X}(n+1)\to \mc{O}_{\mb{B}l_{Z}X}(n)\to \iota_{Z/X*}\mc{O}_{\mb{P}_{Z}(\mc{N}_{Z/X})}(n)
\end{displaymath}
for any integer $n$.
\end{corollary}

\subsection{The derived (extended) Rees-algebra} By (\ref{Rees1}) of \cref{thm:bigdiagram}, we have the following definition
\begin{definition}[Definition 5.13 of Paper B of \cite{Hekking_2022}, the derived Rees-algebra]
  The extended Rees algebra of $Z\to X$ is the $\mb{Z}$-graded quasi-coherent $\mc{O}_{X}[t^{-1}]$-algebra $R_{Z/X}^{ext}:=\mc{O}_{D_{Z/X}}$, where the $\mb{Z}$-grading is induced by the $\mb{G}_{m}$-action on $D_{Z/X}$. We denote
  $$R_{Z/X}^{ext}=\bigoplus_{d\in \mb{Z}}R_{Z/X}^{d},$$
  where $R_{Z/X}^{d}$ is the degree $d$ part of the extended Rees-algebra. We define the Rees algebra as
  $$R_{Z/X}^{\geq 0}:=\bigoplus_{d\geq 0}R_{Z/X}^{d}.$$
  Similarly, we denote
  $$R_{Z/X}^{\leq 0}:=\bigoplus_{d\leq 0}R_{Z/X}^{d}.$$
\end{definition}

\begin{lemma}[Section 6.10 of \cite{hekking2021graded}]
  \label{support1}
  The map $\mc{O}_{X}[t^{-1}]\to R_{Z/X}^{\leq 0}$ is an equivalence.
\end{lemma}

\begin{lemma}[Theorem 5.15 of Paper B of \cite{Hekking_2022} and proof of it]
  The following composition of maps:
  $$\mb{V}_{\mb{P}_{Z}(C_{Z/X})}^{*}(\mc{O}_{\mb{P}_{Z}(C_{Z/X})}(1)) \to D_{Z/X}=\mathrm{Spec} R_{Z/X}^{ext}\to \mathrm{Spec} R_{Z/X}^{\geq 0}$$
  identifies $\mb{V}_{\mb{P}_{Z}(C_{Z/X})}^{*}(\mc{O}_{\mb{P}_{Z}(C_{Z/X})}(1))$ with $\mathrm{Spec} R_{Z/X}^{\geq 0}-X$, and thus induces an equivalence
  $$Bl_{Z/X}\to \mathrm{Proj}_{X}R_{Z/X}.$$
\end{lemma}

\subsection{Derived blow-ups of zero cosections}We first consider several examples of the derived Rees-algebra:
\begin{example}[Example 2.6 of \cite{hekking2022stabilizer}]
  Let $A$ be a connective commutative simplicial algebra over $\mb{C}$. Let $\sigma_{1},\cdots,\sigma_{k}$ be cycles $\sigma_{i}\in \pi_{0}(A)$, and let $B$ be the finite quotient $A/(\sigma_{1},\cdots,\sigma_{k})$. Then
\begin{equation}
     \label{exm:Rees1}
    R_{B/A}^{ext}\cong \frac{A[t^{-1},u_{1},\cdots,u_{n}]}{(u_{1}t^{-1}-\sigma_{1},\cdots, u_{n}t^{-1}-\sigma_{n})},
  \end{equation}
  where the $\mb{G}_{m}$-acts on $R_{B/A}^{ext}$ such that $u_{i}$ are free in homgeneous degree $1$ and homological degree $0$.
\end{example}

\begin{example}[Proposition 2.13 of \cite{hekking2022stabilizer}]
  \label{ex:2}
  Let $A$ be a connective commutative simplicial algebra over $\mathrm{Spec}(\mb{C})$. We are given a sequence of closed embeddings $Z\to U\to \mathrm{Spec} A$, corresponding to
  $$A\to B\to D$$
  where $B=A/(f_{1},\cdots,f_{n})$ for certain $f_{i}\in \pi_{0}A$ and $D=A/(g_{1},\cdots,g_{m})$ for certain $g_{i}\in \pi_{0}A$. Then there exists $\lambda_{ji}\in \pi_{0}A$, such that
  \begin{equation}
    \label{eq:32}
  f_{i}=\sum_{j}\lambda_{ij}g_{j}    
  \end{equation}
  for any $1\leq i\leq n$. Moreover, we have
  \begin{equation}
    R_{Z/U}^{ext}\cong \frac{A[t^{-1},\underline{w}]}{(t^{-1}\underline{w}-\underline{g},\sum_{j}\lambda_{1j}w_{j},\cdots,\sum_{j}\lambda_{nj}w_{j})}
  \end{equation}
  where $\underline{w}=(w_{1},\cdots,w_{n})$ are in homogical degree $0$ and homogeneous degree $1$.
\end{example}

Now we recall \cref{cd:projection} in \cref{ex:215}:
  \begin{equation*}
    \begin{tikzcd}
      \mb{P}_{X}(\mc{F})\ar{r}{i_{0}'} \ar{d}{pr_{\mc{F}}} &  \mb{V}_{\mb{P}_{X}(\mc{F})}(\mc{O}_{\mb{P}_{X}(\mc{F})}(1))\ar{d}{\epsilon_{\mc{F}}} \\
      X \ar{r}{i_{0}} & \mb{V}_{X}(\mc{F})
    \end{tikzcd}
  \end{equation*}
  where $i_{0}$ and $i_{0}'$ are the zero cosections. Reducing to the local case and then applying \cref{ex:2}, we have
  \begin{proposition}
    \label{prop:conj41}    \label{conj:4.1}
    When $\mc{F}$ has perfect-amplitude contained in $[0,1]$ and constant rank $r$, then $$\mb{V}_{\mb{P}_{X}(\mc{F})}(\mc{O}_{\mb{P}_{X}(\mc{F})}(1))\cong \mb{B}l_{X/\mb{V}_{X}(\mc{F})}$$
    and $\epsilon_{\mc{F}}\cong pr_{X/\mb{V}_{X}(\mc{F})}$ is the canonical projection map.
  \end{proposition}

  \begin{corollary}
    \label{cor:ccc} 
     Under the assumption of \cref{prop:conj41}, we have
$$D_{X/\mb{V}_{X}(\mc{F})}\cong \mb{V}_{X}(\mc{F})\times\mb{A}^{1},$$
where the $\mb{G}_{m}$-action is universal on $\mb{V}_{X}(\mc{F})$ and hyperplane on $\mb{A}^{1}$. The projection from $D_{X/\mb{V}_{X}(\mc{F})}$ to $\mb{V}_{X}(\mc{F})\times\mb{A}^{1}$ is induced by
\begin{align*}
  \mb{V}_{X}(\mc{F})\times\mb{A}^{1} & \to \mb{V}_{X}(\mc{F})\times\mb{A}^{1}\\
  (v,t)&\to (tv,t).
\end{align*}
  \end{corollary}
  \begin{remark}
    \cref{prop:conj41} should hold even if only assuming that $\mc{F}$ is a connective quasi-coherent sheaf over $X$, by applying the theory developed in \cite{jiang2022derived,Hekking_2022}. The assumption in \cref{prop:conj41} is already enough for our paper, so we leave the general case to readers with interest.
  \end{remark}

 \section{The Generalized Vanishing Theorem}
 \label{sec:5}
  In this section, we denote $f:Z\to X$ as a closed embedding of derived $S$-Artin stacks which satisfy \cref{ass1}.

\subsection{Local charts}
We first prove a local structure lemma for quasi-smooth derived schemes:
\begin{lemma}
  \label{kuranishi}
  Let $T$ be a derived scheme, and 
  $$Y_{0}\xrightarrow{f}Y_{1}$$ be a closed embeddings of quasi-smooth derived $T$-schemes. Then for any closed point $y\in |Y_{1}|$, there is a Zariski-open neighborhood of $y$, which we denote as $Y_{1}'$, and a sequence of maps of locally free sheaves over a smooth derived $T$-scheme $W$:
  $$V_{1}\xrightarrow{g_{1}}V_{0}\xrightarrow{w_{0}}\mc{O}_{W}$$
  such that $Y_{1}'$ is the derived zero locus of $w_{0}\circ g_{1}$ and $Y_{0}':=Y_{1}'\times_{Y_{1}}Y_{0}$ is the derived zero locus of $w_{0}$. Moreover, we have the following Cartesian diagram:
  \begin{equation}
    \label{eq:refcd}
    \begin{tikzcd}
      Y_{0}' \ar{d}\ar{r} & Y_{1}'\ar{d} \ar{r}& W \ar{d}{i_{w_{0}}}\\
      W \ar{r}{i_{0}} & \mb{V}_{W}(g_{1}) \ar[r] \ar{d} & \mb{V}_{W}(V_{0})\ar{d}{}\\
      & W\ar{r}{i_{0}}  & \mb{V}_{W}(V_{1}).
    \end{tikzcd}
  \end{equation}
\end{lemma}
\begin{proof}
  As it is a Zariski-local question, we can take $T$ and $Y_{1}$ as affine derived schemes. By \cref{lem:2018}, we can moreover assume that there exists a closed embedding
  $$Y_{1}\hookrightarrow \mathrm{Spec} A$$
  where $\mathrm{Spec}A$ is smooth over $T$. Then the closed embedding of $Y_{0}$ into $\mathrm{Spec}\ (A)$ is also quasi-smooth, as the cotangent complex has Tor-amplitude $[-1,0]$. Then still by \cref{lem:2018}, we can assume that (locally)
  $$Y_{1}= \mathrm{Spec} A/(l_{0},\cdots l_{n}), \ Y_{0}=\mathrm{Spec}A/(g_{0},\cdots,g_{m}),$$
  for certain $l_{i},g_{j}\in \pi_{0}A$. The construction of $f_{1}$ follows from \cref{ex:2}. 
\end{proof}

\begin{remark}
  The tuple $(Y_{i}',W,V_{i}^{\vee},w_{i}^{\vee})$ is also called a Kuranishi chart of $Y_{i}$. (see \cite{MR3904157}\cite{MR4379049})
\end{remark}

\begin{lemma}[Lemma C.7 of \cite{yuzhaoderived}, also see Proposition 3.13 of \cite{hekking2022stabilizer}]
  \label{lemA421}
  let $W$ be a derived scheme with morphisms of locally free sheaves:
   $$V_{1}\xrightarrow{g_{1}}V_{0}\xrightarrow{w_{0}}\mc{O}_{W}$$
and we denote $w_{1}:=w_{0}\circ g_{1}$. Let $Y_{0}$ and $Y_{1}$ be the derived zero locus of $w_{0}$ and $w_{1}$ respectively. We consider the projection morphism:
$$pr_{V_{0}}:\mb{P}_{W}(V_{0})\to W.$$
Let
$$\bar{w}_{0}:L_{\mb{P}_{W}(V_{0})/W}\otimes \mc{O}_{\mb{P}_{W}(V_{0})}(1)\to \mc{O}_{\mb{P}_{W}(V_{0})}$$ be the morphism of $pr_{V_{0}}^{*}w_{0}$ restricted to $L_{\mb{P}_{W}(V_{0})/W}\otimes \mc{O}_{\mb{P}_{W}(V_{0})}(1)$, and
$$\bar{g_{1}}: pr_{V_{0}}^{*}V_{1}\otimes \mc{O}_{\mb{P}_{W}(V_{0})}(-1)\to \mc{O}_{\mb{P}_{W}(V_{0})}$$ as the morphism $\mc{O}_{\mb{P}_{W}(V_{0})}(-1)\otimes (\rho_{V_{0}}\circ pr_{V_{0}}^{*}g_{1})$. Then the derived blow-ups

$$\mb{B}l_{Y_{0}/W} \text{ and } \mb{B}l_{Y_{0}/Y_{1}}$$
are the derived zero locus of
$$\bar{w_{0}} \text{ and } \bar{w_{0}}+\bar{g_{1}}$$
over $\mb{P}_{W}(V_{0})$ respectively.
\end{lemma}
\begin{proof}
  We have the following Cartesian diagram:
  \begin{equation*}
        \begin{tikzcd}
      Y_{0} \ar{d}\ar{r} & Y_{1}\ar{d} \ar{r}& W \ar{d}{i_{w_{0}}}\\
      W \ar{r}{i_{0}} & \mb{V}_{W}(g_{1}) \ar[r] \ar{d} & \mb{V}_{W}(V_{0})\ar{d}{}\\
      & W\ar{r}{i_{0}}  & \mb{V}_{W}(V_{1}).
    \end{tikzcd}
  \end{equation*}
  By the functorial property of derived blow-ups, we can replace
  $$Y_{0},Y_{1},W \quad \text{    by    }\quad  W,\mb{V}_{W}(g_{1}),\mb{V}_{W}(V_{0})$$ and then  \cref{lemA421} follows from \cref{ex:2} and \cref{prop:conj41}.
\end{proof}

\begin{theorem}
  \label{cor:ddd}\label{structurethm}Assuming \cref{ass1}, we have that
  \begin{enumerate}
      \item the blow-up $\mb{B}l_{Z/X}$ is also a quasi-smooth derived Artin stack over $S$;
      \item the cotangent complex $L_{\mb{B}l_{Z/X}/S}$ and $L_{X/S}$ have the same rank;
      \item the projection $pr_{Z/X}$ is proper.
      \item the open embedding 
      $$j_{Z/X}:\mb{V}_{\mb{B}l_{Z/X}}^{*}(
      \mc{O}_{\mb{B}l_{Z/X}}(1))\to D_{Z/X} $$
      in \cref{06052132} is quasi-compact and locally of finite presentation.
  \end{enumerate}
\end{theorem}
\begin{proof}
  By the fppf descent, we can assume that $Z$ and $X$ are both derived schemes. Then locally we can replace $Z$ and $X$ by $Y_{0}'$ and $Y_{1}'$ in \cref{eq:refcd}. Thus \cref{cor:ddd} follows from \cref{prop:conj41} and \cref{cor:ccc}.
\end{proof}

  \subsection{The canonical morphism $\phi_{\geq d}^{f}:pr_{f*}(\mc{O}_{\mb{B}l_{Z/X}}(d))\to R_{Z/X}^{d}$}
  \label{sec:4.1}
  We recall the $\mb{G}_{m}$-equivariant Cartesian diagram \cref{eq:051734}
   \begin{equation*}
        \begin{tikzcd}
      Z\times \{0\} \ar{r}\ar{d} & Z\times \mb{A}^{1}\ar{d} & Z\times \mb{G}_{m}\ar{d}\ar{l} \\
      \mb{V}_{Z}(C_{Z/X})\ar{r} & D_{Z/X} & X\times \mb{G}_{m}\ar{l}, \\
    \mb{V}_{E_{Z/X}}^{*}(\mc{O}_{\mb{P}_{Z}(C_{Z/X})}(1)|_{E_{Z/X}})\ar{u} \ar{r} & \mb{V}_{\mb{B}l_{Z/X}}^{*}(\mc{O}_{\mb{B}l_{Z/X}}(1)) \ar{u}{j_{Z/X}}   & (Z-X)\times \mb{G}_{m} \ar{l}\ar{u} .\\
    \end{tikzcd}
  \end{equation*}
  We notice that
  \begin{align*}
    (d_{Z/X}\circ j_{Z/X})_{*}\mc{O}_{\mb{V}_{\mb{P}_{Z}(C_{Z/X})}^{*}(\mc{O}_{\mb{P}_{Z}(C_{Z/X})}(1))}&\cong \bigoplus_{d\in \mb{Z}}pr_{f*}\mc{O}_{\mb{B}l_{Z/X}}(d),\\
    d_{Z/X*}\mc{O}_{D_{Z/X}}&\cong \bigoplus_{d}R_{Z/X}^{d}
    \end{align*}
  as $\mb{G}_{m}$-equivariant quasi-coherent $\mc{O}_{X}[t^{-1}]$-modules, where $pr_{f*}\mc{O}_{\mb{B}l_{Z/X}}(d)$ and $R_{Z/X}^{d}$ both have homogeneous degree $d$. Moreover, the $t^{-1}$-action
  $$pr_{f*}\mc{O}_{\mb{B}l_{Z/X}}(d+1)\xrightarrow{t^{-1}}pr_{f*}\mc{O}_{\mb{B}l_{Z/X}}(d)$$
  is the pushforward of the canonical map of the universal Cartier divisor:
  $$\mc{O}_{\mb{B}l_{Z/X}}(d+1)\to \mc{O}_{\mb{B}l_{Z/X}}(d).$$
  As the open embedding
  $$j_{Z/X}:\mb{V}_{\mb{B}l_{Z/X}}^{*}(\mc{O}_{\mb{B}l_{Z/X}}(1))\to D_{Z/X}$$
  is quasi-compact and locally of finite presentation by \cref{cor:ddd}, it induces a canonical map of $\mc{O}_{X}[t^{-1}]$ quasi-coherent sheaves
  $$d_{Z/X*}\mc{O}_{D_{Z/X}}\to (d_{Z/X}\circ j_{Z/X})_{*}\mc{O}_{\mb{V}_{\mb{P}_{Z}(C_{Z/X})}^{*}(\mc{O}_{\mb{P}_{Z}(C_{Z/X})}(1))},$$
  which is represented by
  \begin{equation}
    \label{eq:phif}
  \phi^{f}:R_{Z/X}^{ext}\to \bigoplus_{d\in \mb{Z}}pr_{f*}\mc{O}_{\mb{B}l_{Z/X}}(d).
  \end{equation}
  Moreover, $\phi^{f}$ decomposes into   $\bigoplus_{d\in \mb{Z}}\phi_{\geq d}^{f}$, where
  \begin{equation}
    \label{eq:phid}
  \phi_{\geq d}^{f}:R_{Z/X}^{d}\to pr_{f*}\mc{O}_{\mb{B}l_{Z/X}}(d)    
  \end{equation}
  is the homogeneous degree $d$ part of $\phi^{f}$.

  By the Cartesian diagram \cref{eq:051734}, we have the following commutative diagram of fiber sequences of quasi-coherent $\mc{O}_{X}[t^{-1}]$-modules over $X$
  \begin{equation}
    \label{eq:phin}
      \begin{tikzcd}
      R_{Z/X}^{ext}\ar{r}{t^{-1}} \ar{d}{\phi^{f}} & R_{Z/X}^{ext}\ar{r}\ar{d}{\phi^{f}} & \bigoplus_{d\in \mb{Z}}f_{*}(S^{d}(C_{Z/X}))\ar{d}{f_{*}\phi_{C_{Z/X}}} \\
      \bigoplus_{d\in \mb{Z}} pr_{f*}\mc{O}_{\mb{B}l_{Z/X}}(d)\ar{r}{t^{-1}} & \bigoplus_{d\in \mb{Z}} pr_{f*}\mc{O}_{\mb{B}l_{Z/X}}(d) \ar{r}& \bigoplus_{d\in \mb{Z}}f_{*}pr_{C_{Z/X}*}(\mc{O}_{\mb{P}_{Z}(C_{Z/X})}(d)),
    \end{tikzcd}
  \end{equation}
  where $\phi_{C_{Z/X}}$ is defined by the generalized Serre theorem \cref{thm:serre2}.

  \begin{lemma}[Remark 5.10.7 of \cite{hekking2021graded}]
    \label{support}
    The canonical morphism
    $$\mc{O}_{X}\to f_{*}\mc{O}_{Z}$$ is represented by the morphism $\mc{O}_{X}\cong R^{0}_{Z/X}\to f_{*}\mc{O}_{Z}$ in \cref{eq:phin} 
      \end{lemma}

\subsection{The generalized vanishing theorem \RN{1}}
\label{sec:cosec}
Now we formulate the generalized vanishing theorem \RN{1}:
\begin{theorem}
  \label{isomorphism2}
  There exists a canonical morphism of $\mb{Z}$-graded quasi-coherent $\mc{O}_{X}[t^{-1}]$-modules
  $$\phi^{f}:R_{Z/X}^{ext}\to W_{Z/X}$$
  with the following commutative diagram (up to equivalence) where all the horizontal and vertical arrows are fiber sequences:
  \begin{equation*}
      \begin{tikzcd}
      R_{Z/X}^{ext}\ar{r}{t^{-1}} \ar{d}{\phi^{f}} & R_{Z/X}^{ext}\ar{r}\ar{d}{\phi^{f}} & \mf{H}_{\bullet}^{f+} \ar{d}{f_{*}\phi_{C_{Z/X}}} \\
      W_{Z/X}\ar{r}{t^{-1}}\ar{d} & W_{Z/X} \ar{d}\ar{r}& f_{*}pr_{C_{Z/X}*}(\mc{O}_{\mb{P}_{Z}(C_{Z/X})}(\bullet)) \ar{d}{f_{*}\psi_{C_{Z/X}}}\\
      cofib(\phi^{f})\ar{r}{t^{-1}} & cofib(\phi^{f})\ar{r} & \mf{H}_{\bullet}^{f-},
    \end{tikzcd}
  \end{equation*}
  where
  \begin{itemize}
  \item $cofib(\phi^{f})$ is the cofiber of $\phi^{f}$:
  \item we denote 
  \begin{align*}
   &   \mf{H}_{b}^{f+}:=f_{*}S^{b}(C_{Z/X})\\
   & \mf{H}_{b}^{f-}:=f_{*}(S^{-r-b}(C_{Z/X})^{\vee}\otimes det(C_{Z/X})^{\vee})[1-r]
  \end{align*}
  for $b\in \mb{Z}$. We denote $\mf{H}_{\bullet}^{f+}:=\oplus_{d\in \mb{Z}}\mf{H}_{b}^{f+}$ and make similar notations for $pr_{C_{Z/X}*}(\mc{O}_{\mb{P}_{Z}(C_{Z/X})}(\bullet))$ and $\mf{H}_{\bullet}^{f-}$.
  \item $\phi_{C_{Z/X}}$ and $\psi_{C_{Z/X}}$ is defined in Jiang's generalized Serre theorem (\cref{thm:serre2}).
  \end{itemize}
 Moreover, when $d>-r$, the degree $d$ part of $\phi^{f}$, which we denote as $\phi^{f}_{\geq d}$, is an equivalence. 
\end{theorem}
\begin{proof} Most of the theorem follows from \cref{eq:phin} and Jiang's generalized theorem \cref{thm:serre2}. The only thing we still need to prove is that  when $d>-r$, $\phi^{f}_{\geq d}$ is an equivalence. 

First, we consider the special case that
$$X\cong \mb{V}_{Z}(\mc{F})$$
(with trivial $\mb{G}_{m}$-action), where
\begin{itemize}
\item $\mc{F}$ is a quasi-coherent sheaf with perfect-amplitude contained in $[0,1]$ and constant rank $r$ over $Z$;
\item  $f:Z\to X$ is the zero cosection $i_{0,\mc{F}}$.
\end{itemize}
Then by \cref{prop:conj41} and \cref{cor:ccc}, we have
$$\mb{B}l_{Z/X}\cong\mb{V}_{\mb{P}_{Z}(\mc{F})}(\mc{O}_{\mb{P}_{Z}(\mc{F})}(1)), \quad D_{Z/X}\cong \mb{V}_{Z}(\mc{F})\times\mb{A}^{1},$$
where the $\mb{G}_{m}$-action on $D_{Z/X}$ is induced by the hyperplane $\mb{G}_{m}$-action. The projection from $D_{Z/X}$ to $X\times \mb{A}^{1}$ is induced by
\begin{align*}
  \mb{V}_{Z}(\mc{F})\times\mb{A}^{1} & \to \mb{V}_{Z}(\mc{F})\times\mb{A}^{1}\\
  (v,t)&\to (tv,t).
\end{align*}
Moreover, we have
$$R_{Z/X}^{d}\cong \bigoplus_{n\geq d}S^{n}\mc{F},\quad pr_{f*}\mc{O}_{\mb{B}l_{Z/X}}(d)\cong \bigoplus_{n\geq d}pr_{\mc{F}*}\mc{O}_{\mb{P}_{Z}(\mc{F})}(n),$$
where the latter one is a $S_{Z}^{*}\mc{F}$ module and hence regarded as a quasi-coherent sheaf over $\mb{V}_{Z}(\mc{F})$. Moreover, $\phi_{\geq d}^{f}$ is represented by
$$\bigoplus_{n\geq d}\phi_{n,\mc{F}}:\bigoplus_{n\geq d}S^{n}\mc{F}\to  \bigoplus_{n\geq d}pr_{\mc{F}*}\mc{O}_{\mb{P}_{Z}(\mc{F})}(n).$$
By Jiang's generalized Serre's theorem (\cref{thm:serre2}), the map
\begin{equation}
  \label{isomorphism1}
  \oplus_{n\geq d}\phi_{n,\mc{F}}\text{ is an equivalence when } d>-r.
\end{equation}

For the general case, by the fppf descent of quasi-coherent sheaves, we can assume that both $X$ and $S$ are affine derived schemes. Then by \cref{kuranishi}, we can moreover assume that there is a sequence of  maps of locally free sheaves over a smooth derived $S$-scheme $W$:
  $$V_{1}\xrightarrow{g_{1}}V_{0}\xrightarrow{w_{0}}\mc{O}_{W}$$
  such that
  $X$ is the derived zero locus of $w_{0}\circ g_{1}$ and $Z$ is the derived zero locus of $w_{0}$. Thus we have the Cartesian diagram:
  \begin{equation*}
    \begin{tikzcd}
      Z \ar{d}\ar{r} & X\ar{d} \ar{r}& W \ar{d}{i_{w_{0}}}\\
      W \ar{r}{i_{0}} & \mb{V}_{W}(g_{1}) \ar[r] \ar{d} & \mb{V}_{W}(V_{0})\ar{d}{}\\
      & W\ar{r}{i_{0}}  & \mb{V}_{W}(V_{1}).
    \end{tikzcd}
  \end{equation*}
  We have $rank(g_{1})=r$. Hence the equivalence of $\phi_{\geq d}^{f}$ when $d>-codim_{Z/X}$ follows from the left square of the above diagram, the fact that $\phi_{\geq d}^{f}$ is compatible with the pull back and \cref{isomorphism1}.
\end{proof}

\subsection{The generalized vanishing theorem \RN{2}}
Given $a\leq b$, we consider the following map of quasi-coherent sheaves:
\begin{align*}
  \phi_{\geq b}^{f}:R_{Z/X}^{b}\to pr_{f*}\mc{O}_{\mb{B}l_{Z/X}}(b),\quad  t^{a-b}:R_{Z/X}^{b}\to R_{Z/X}^{a}
\end{align*}
which are defined in \cref{sec:4.1}.
\begin{definition}
  We define the quasi-coherent sheaves $\mf{W}_{a,b}^{f}$ on $X$ as the cofiber of
  $$\phi_{\geq b}^{f}\oplus t^{a-b}:R_{Z/X}^{b}\to  pr_{f*}\mc{O}_{\mb{B}l_{Z/X}}(b)\oplus R_{Z/X}^{a}.$$

  For any $(a_{1},b_{1})\geq (a_{2},b_{2})$, we consider the map
  $$\mf{w}^{f}_{(a_{1},b_{1}),(a_{2},b_{2})}:\mf{W}^{f}_{a_{1},b_{1}}\to \mf{W}^{f}_{a_{2},b_{2}}$$
  as the map induced by the following three maps:
  \begin{gather*}
    t^{b_{2}-b_{1}}:R_{Z/X}^{b_{1}}\to R_{Z/X}^{b_{2}},\quad    t^{a_{2}-a_{1}}:R_{Z/X}^{a_{1}}\to R_{Z/X}^{a_{2}},\\
    t^{b_{2}-b_{1}}:pr_{f*}\mc{O}_{\mb{B}l_{Z/X}}(b_{1})\to pr_{f*}\mc{O}_{\mb{B}l_{Z/X}}(b_{2}).
  \end{gather*}
\end{definition}

\begin{definition}
  \label{def:categorical}
  For any integer $b$, we recall the definition
  $$\mf{H}_{b}^{f+}\cong f_{*}S^{b}(C_{Z/X}), \quad \mf{H}_{b}^{f-}\cong f_{*}(S^{-r-b}(C_{Z/X})^{\vee}\otimes det(C_{Z/X})^{\vee})[1-r].$$
 If $a<b$, we define the map
  $$h_{a,b}^{f+}:\mf{W}_{a,b}^{f}\to \mf{H}_{a}^{f+}$$
  as the map induced by   $R_{Z/X}^{a}\to f_{*}S^{a}(C_{Z/X})$ in \cref{eq:phin}
  (and $0$ in the other two summands).

  If $a\leq  b$, we define the map
  $$h_{b}^{f-}:\mf{W}_{a,b}^{f}\to \mf{H}_{b}^{f-}$$
  as induced by the composition
  $$pr_{Z/X*}\mc{O}_{\mb{B}l_{Z/X}(b)}\to f_{*}pr_{C_{Z/X}*}\mc{O}_{\mb{P}_{Z}(C_{Z/X})}(b)\xrightarrow{f_{*}(\psi_{b,C_{Z/X}})} \mf{H}_{b}^{f-}$$
  in \cref{eq:phin}(and $0$ in the other two summands).
\end{definition}

Now we formulate the generalized vanishing theorem \RN{2}:
\begin{theorem}
  \label{thm:main}
  The quasi-coherent sheaves $\mf{W}_{a,b}^{f}$ and $\mf{H}_{b}^{f\pm}$ over $X$ have the following properties:
  \begin{enumerate}
 \item 
  $$\mf{w}_{(a_{2},b_{2}),(a_{3},b_{3})}^{f}\circ\mf{w}_{(a_{1},b_{1}),(a_{2},b_{2})}^{f}\cong \mf{w}_{(a_{1},b_{1}),(a_{3},b_{3})}^{f}$$
  when all the above morphisms are well defined and $\mf{w}_{(a,b),(a,b)}^{f}\cong id$. 
  \item For any $a$, we have
    $$\mf{W}_{a,a}^{f}\cong pr_{f*}\mc{O}_{\mb{B}l_{Z/X}}(a)$$
    and for any $a\leq b$, the morphism $\mf{w}_{(b,b),(a,a)}^{f}$ is represented by the push forward of the canonical map
    $$\mc{O}_{\mb{B}l_{Z/X}}(b)\to \mc{O}_{\mb{B}l_{Z/X}}(a)$$
\item For any $b\geq -r+1$, we have
    $$\mf{W}_{a,b}^{f}\cong R_{Z/X}^{a},$$
    and for any $b\geq b'\geq -r+1$, the morphism $\mf{w}_{(a,b),(a',b')}^{f}$ is represented by
    $$t^{a'-a}:R_{Z/X}^{a}\to R_{Z/X}^{a'}.$$
\item  For any $a<b$, we have the commutative diagram (up to equivalence) where all the rows and columns are fiber sequences:
\begin{equation*}
    \begin{tikzcd}[column sep=2cm]
        \mf{W}_{a+1,b+1}^{f}\ar{r}{\mf{w}_{(a+1,b+1),(a+1,b)}^{f}}\ar{d}{\mf{w}_{(a+1,b+1),(a,b+1)}^{f}} & \mf{W}_{a+1,b}^{f}\ar{d}{\mf{w}_{(a+1,b),(a,b)}^{f}} \ar{r}{\mf{H}_{a+1,b}^{f-}} & \mf{H}_{b}^{f-}\ar{d}{id}\\
        \mf{W}_{a,b+1}^{f} \ar{r}{\mf{w}_{(a,b+1),(a,b)}^{f}}\ar{d}{\mf{H}_{a,b+1}^{f+}} & \mf{W}_{a,b}^{f} \ar{r}{\mf{H}_{a,b}^{f-}} \ar{d}{\mf{H}_{a,b}^{f+}}& \mf{H}_{b}^{f-} \\
        \mf{H}_{a}^{f+}\ar{r}{id} &  \mf{H}_{a}^{f+}
    \end{tikzcd}
\end{equation*}
\item For any $a\in \mb{Z}$, we have the commutative diagram (up to equivalence) where all the rows and columns are fiber sequences:
   \begin{equation*}
      \begin{tikzcd}[column sep=2cm]
        \mf{W}_{a+1,a+1}^{f}\ar{r}{id}\ar{d}{\mf{w}_{(a+1,a+1),(a,a+1)}} & \mf{W}_{a+1,a+1}^{f} \ar{d}{\mf{w}_{(a+1,a+1),(a,a)}}\\
        \mf{W}_{a,a+1}^{f}\ar{r}{\mf{w}_{(a-1,a),(a,a)}} \ar{d}{h_{a,a+1}^{+}} & \mf{W}_{a,a}^{f} \ar{r}{h_{a,a}^{-}}\ar{d} &  \mf{H}_{a}^{f-}\ar{d} \\
        \mf{H}_{a}^{f+}\ar{r}{f_{*}(\phi_{a,C_{Z/X}})} & f_{*}pr_{C_{Z/X}*}\mc{O}_{\mb{P}_{Z}(C_{Z/X})}(a) \ar{r}{f_{*}(\psi_{a,C_{Z/X}})} & \mf{H}_{a}^{f-}.
      \end{tikzcd}
    \end{equation*}
  \item \label{if6}
    For any $d\geq max\{-r+1,0\}$, the canonical map
    $$\mc{O}_{X}\to pr_{f*}(\mc{O}_{\mb{B}l_{Z/X}})$$
    is represented by $\mf{w}_{(0,d),(0,0)}$, which is an equivalence if $r>0$. 
\item \label{if7} The morphism
    $$\mc{O}_{X}\to f_{*}\mc{O}_{Z}$$
    is represented by $h_{0,d}^{f+}$ for any $d\geq max\{-r+1,0\}$.
\end{enumerate}
\end{theorem}
\begin{proof}
  All except the item (\ref{if6}) and (\ref{if7}) follow from the definitions and the diagram \cref{eq:phin}. The item (\ref{if6}) follows from \cref{support1} and \cref{isomorphism2}. The item (\ref{if7}) follows from \cref{support}.
\end{proof}

\appendix

\section{The resolution of the diagonal of derived blow-ups }
\label{sec:A}
Let $f:Z\to X$ be a quasi-smooth closed embedding of derived Artin stacks, such that the co-normal complex $C_{Z/X}$ is a rank $r$ locally free sheaf over $Z$. The universal strict virtual Cartier divisor diagram
\begin{equation*}
  \begin{tikzcd}
  \mb{P}_{Z}(C_{Z/X})\ar{r}\ar{d} & \mb{B}l_{Z/X}\ar{d} \\
  Z\ar{r} & X    
  \end{tikzcd}
\end{equation*}
induces a closed embedding of derived stacks:
$$\beta_{Z/X}:\mb{P}_{Z}(C_{Z/X})\times_{Z}\mb{P}_{Z}(C_{Z/X})\to  \mb{B}l_{Z/X}\times_{X}\mb{B}l_{Z/X}.$$
Moreover, we have the following diagram of the virtual Cartier divisor
\begin{equation}
  \label{eq:resolution}
  \begin{tikzcd}
    \mb{P}_{Z}(C_{Z/X})\ar{d}{\Delta_{pr_{C_{Z/X}}}} \ar{r} & \mb{B}l_{Z/X}\ar{d}{\Delta_{pr_{Z/X}}}\\
    \mb{P}_{Z}(C_{Z/X})\times_{Z}\mb{P}_{Z}(C_{Z/X})\ar{r}{\beta_{Z/X}}&  \mb{B}l_{Z/X}\times_{X}\mb{B}l_{Z/X},
  \end{tikzcd}
\end{equation}
where $\Delta_{pr_{C_{Z/X}}}$ and $\Delta_{pr_{Z/X}}$ are the diagonal maps of $pr_{C_{Z/X}}$ and $pr_{Z/X}$ respectively. Our new observation is that
\begin{theorem}
  \label{thm:resolution}
  The diagram \cref{eq:resolution} is the universal strict virtual Cartier divisor of $\beta_{Z/X}$.
\end{theorem}
\begin{proof}
  By the fppf descent, \cref{131666} and the fact that derived blow-ups are stable under the pull-backs, we only need to consider the case that $Z\cong \{0\}$ and $X\cong \mb{A}^{r}$. We denote $\mb{P}_{0}^{r}$, $\mb{P}^{r}_{1}$ and $\mb{P}^{r}_{2}$ as three copies of $\mb{P}^{r}$. Over $\mb{A}^{r}\times \mb{P}_{1}^{r}\times \mb{P}_{2}^{r}$, we consider the following morphism of locally free sheaves:
  $$L_{\mb{P}^{r}_{1}}\otimes \mc{O}_{\mb{P}^{r}_{1}}(-1)\oplus L_{\mb{P}^{r}_{2}}\otimes \mc{O}_{\mb{P}^{r}_{2}}(-1)\xrightarrow{e_{1}+e_{2}} \mc{O}^{r}\xrightarrow{u} \mc{O}$$
  where $e_{i}$ are the pullbacks of Euler fiber sequences and 
  $$u(x_1,\cdots,x_n)=x_1+\cdots +x_n.$$  Then
  $\mb{P}_{1}^{r}\times \mb{P}_{2}^{r}$ is the derived zero locus of $u$ and $\mb{V}_{\mb{P}_{1}^{r}}(\mc{O}_{\mb{P}_{1}^{r}}(1))\times_{\mb{A}^{r}}\mb{V}_{\mb{P}_{2}^{r}}(\mc{O}_{\mb{P}_{2}^{r}}(1))$ is the derived zero locus of $u\circ (e_{1}+e_{2})$. Thus by \cref{lemA421}, $\mb{B}l_{\beta_{Z/X}}$ is the derived locus of the following co-section over $\mb{V}_{\mb{P}_{0}^{r}}(\mc{O}_{\mb{P}_{0}^{r}}(1))\times  \mb{P}_{1}^{r}\times \mb{P}_{2}^{r}$:
  $$L_{\mb{P}^{r}_{1}}\otimes \mc{O}_{\mb{P}^{r}_{1}}(-1)\oplus L_{\mb{P}^{r}_{2}}\otimes \mc{O}_{\mb{P}^{r}_{2}}(-1)\xrightarrow{\rho_{0}\circ (e_{1}+e_{2})}\mc{O}_{\mb{P}_{0}^{r}}(1)$$
  where $\rho_{0}$ is the tautological morphism over $\mb{P}_{0}^{r}$, which is  
  $\mb{V}_{\mb{P}^{r}}(\mc{O}_{\mb{P}^{r}}(1))$
  by Beilinson's resolution \cite{beilinson1978coherent} of the diagonal of $\mb{P}^{r}$.
\end{proof}

Now we compute the co-normal complex of $\beta_{Z/X}$. As the co-normal complex of $\mb{P}_{Z}(C_{Z/X})$ in $\mb{B}l_{Z/X}$ is
$\mc{O}_{\mb{P}_{Z}(C_{Z/X})}(1)$, the co-normal complex of  $\beta_{Z/X}$ is
$$C_{Z/X}\xrightarrow{\rho_{1}\oplus \rho_{2}}\mc{O}_{\mb{P}_{Z}(C_{Z/X}^{1})}(1)\oplus \mc{O}_{\mb{P}_{Z}(C_{Z/X}^{2})}(1)$$
where $C_{Z/X}^{1}$ and $C_{Z/X}^{2}$ are two copies of $C_{Z/X}$, and we abuse the notation to denote $C_{Z/X},\mc{O}_{\mb{P}_{Z}(C_{Z/X}^{1})}(1)$ and $\mc{O}_{\mb{P}_{Z}(C_{Z/X}^{2})}(1)$ as the pull-back of $C_{Z/X},\mc{O}_{\mb{P}_{Z}(C_{Z/X}^{1})}(1)$ and $\mc{O}_{\mb{P}_{Z}(C_{Z/X}^{2})}(1)$ from $Z,\mb{P}_{Z}(C_{Z/X}^{1})$ and $\mb{P}_{Z}(C_{Z/X}^{2})$ to $\mb{P}_{Z}(C_{Z/X}^{1})\times_{Z}\mb{P}_{Z}(C_{Z/X}^{2})$ respectively. It induces the equivalence
$$C_{\beta_{Z/X}}\cong L_{\mb{P}_{Z}(C_{Z/X}^{1})/Z}\otimes \mc{O}_{\mb{P}_{Z}(C_{Z/X}^{1})}(-1)\to \mc{O}_{\mb{P}_{Z}(C_{Z/X}^{2})}(1).$$
Now we apply the categorical comparison theorem \cref{thm:main} to $\beta_{Z/X}$.
We notice that the codimension of $\beta_{Z/X}$ is $-(r-2)$. 
\begin{theorem}
  \label{thm:fil}
  We have quasi-coherent sheaves $$\mf{W}_{i,r-1}\in \mathrm{Perf}(\mb{B}l_{Z/X}\times\mb{B}l_{Z/X}),\quad 0\leq i\leq r-1$$
  such that
  \begin{enumerate}
      \item $$\mf{W}_{0,r-1}\cong \mc{O}_{\mb{B}l_{Z/X}\times_{X}\mb{B}l_{Z/X}},\quad \mf{W}_{r-1,r-1}\cong \Delta_{\mb{B}l_{Z/X}*}(\mc{O}_{\mb{B}l_{Z/X}}(r-1))$$
      where $\Delta_{\mb{B}l_{Z/X}}$ is the diagonal embedding of $\mb{B}l_{Z/X}$;
      \item we have fiber sequences:
  $$\mf{W}_{i+1,r-1}\to \mf{W}_{i,r-1}\to \bar\beta_{Z/X*} S^{i}(C_{\beta_{Z/X}})$$
  for $0\leq i\leq r-2$, where $\bar{\beta}_{Z/X}$ is the closed embedding of $\mb{P}_{Z}(C_{Z/X})\times_{Z}\mb{P}_{Z}(C_{Z/X})$ into  $\mb{B}l_{Z/X}\times\mb{B}l_{Z/X}$.
  \end{enumerate}
\end{theorem}
By regarding all the perfect complexes of \cref{thm:fil} as Fourier-Mukai transforms, we induce another proof for the full generation result of Khan \cite[Theorem C]{MR4149835} and moreover, Orlov \cite{MR1208153}  for smooth varieties.

\section{Desingulization of Quasi-smooth Derived Schemes}
In this section, we study the desingularization of quasi-smooth derived schemes. The main theorem of this section is 
\begin{theorem}[Desingulization Theorem]
  \label{thm:main2}
  Let $X$ be a non-empty quasi-smooth derived scheme. We assume that there is a closed embedding $p:X\to S$ where $S$ is a smooth variety. Then there exists a collection of quasi-smooth derived schemes $\{X_{i}\}_{0\leq i\leq n}$ and a collection of smooth schemes $\{Z_{i}\}_{0\leq  i<n}$ with closed embeddings $f_{i}:Z_{i}\to X_{i}$ for $0\leq i<n$, where $n$ is a positive integer, such that
  \begin{enumerate}
  \item $X_{0}\cong X$ and $X_{n}\cong \emptyset$;
  \item $X_{i}\cong \mb{B}l_{Z_{i-1}/X_{i-1}}$ for $1\leq i\leq n$.
  \item for any $0\leq  i<n$, $X_{i}$ admits a closed embedding $p_{i}:X_{i}\to S_{i}$ where $S_{i}$ is a smooth variety. 
  \end{enumerate}
\end{theorem}

Together with \cref{eq:noref}, we have the following theorem for the virtual fundamental class of quasi-smooth derived schemes:
\begin{theorem}[Approximation Theorem]
  \label{conj:approx}
  Let $X$ be a quasi-smooth derived scheme with a closed embedding into a smooth ambient variety. Let $vdim(X)$ be the virtual dimension of $X$. Then there exists a collection $\{(Z_{i},p_{i},\mc{F}_{i},r_{i})\}|_{1\leq i\leq n}$ such that for any $1\leq i\leq n$,
  \begin{enumerate}
  \item $r_{i}$ is non-positive integer;
  \item $Z_{i}$ is a smooth scheme such that $dim(Z_{i})=r_{i}+vdim(X)$;
  \item  $p_{i}:Z_{i}\to X$ is a proper morphism;
  \item  $\mc{F}_{i}$ is a perfect complex over $Z_{i}$ with tor-amplitude $[0,1]$ and $rank(\mc{F}_{i})=r_{i}$
  \end{enumerate}
  such that
  \begin{equation}
    \label{eq:formula}
    [\mc{O}_{X}]=\sum_{i=1}^{n}(-1)^{r_{i}}p_{i*}(det(\mc{F}_{i})^{-1}\sum_{j=0}^{-r_{i}}S^{j}(\mc{F}_{i}^{\vee})).
  \end{equation}
\end{theorem}

\subsection{Local Model of Derived Blow-ups}

Given a morphism of smooth schemes $\sigma: M\to N$ and let $\mc{I}$ be an ideal sheaf of $N$, we denote $\sigma^{*}(\mc{I})$ as the image of the pull back of $\mc{I}$ in $\mc{O}_{M}$, which is an ideal sheaf of $N$. Given a regular divisor $D\subset N$, we denote $\mc{I}(D)$ as the ideal sheaf of $D$. If $\mc{I}\subset\mc{I}(D)$, where $\mc{I}$ is an ideal sheaf of $N$, then $\mc{I}(D)^{-1}\otimes \mc{I}$ is also an ideal sheaf of of $N$, which we denote as $\mc{I}(D)^{-1}\mc{I}$.

Let $f:Z\to X$ be a closed embedding of derived schemes such that $X$ is smooth. We denote $\mc{I}(Z,X)$ as the ideal sheaf of $\pi_{0}(Z)$ in $X$.

Let
$$Z\xrightarrow{f}X\xrightarrow{p}S$$
be closed embeddings of quasi-smooth derived schemes such that $Z$ and $S$ are both smooth. We denote $\bar{X}:=X\times_{Z}\mb{B}l_{Z/S}$. Then we have the following commutative diagram:
\begin{equation}
  \label{dglocal}
  \begin{tikzcd}
    E_{Z/X}\ar{r}\ar{d} & \mb{B}l_{Z/X}\ar{d} \\
    E_{Z/S}\ar{r}\ar{d} & \bar{X}\ar{d} \ar{r} & \mb{B}l_{Z/S}\ar{d}\\
    Z \ar{r}{f}  & X \ar{r}{p} & S.
  \end{tikzcd}
\end{equation}
\begin{lemma}
  \label{0524B} \label{cor:0306}
  The upper left square of \cref{dglocal} is a strict universal virtual divisor, where all the arrows in the upper left square are closed embeddings. Moreover, under the closed embedding $\mb{B}l_{Z/X}\to \mb{B}l_{Z/S}$, we have
  $$\mc{I}(\mb{B}l_{Z/X},\mb{B}l_{Z/S})\cong \mc{I}(E_{Z/S})^{-1}pr_{Z/X}^{*}(\mc{I}(X,S)).$$
\end{lemma}
\begin{proof}
  As it is a local question, we can reduce $Z, X,S$ to the setting of $R,S,T$ in \cref{lemA421}. Then \cref{0524B} follows from \cref{lemA421}.
\end{proof}

\begin{remark}
  Actually, \cref{0524B} holds if all $Z,X,S$ are quasi-smooth derived schemes and $p$ is a quasi-smooth closed embedding, through a more detailed study of the Kuranishi charts. We do not need this generalization in our paper and thus left it as an exercise for readers with interest.
\end{remark}

The following is a simplified version of the Hironaka principalization theorem:
\begin{theorem}[Hironaka \cite{MR0199184,MR0498562}, Bierstone-Milman \cite{MR1001853}, Villamayor \cite{MR985852,MR1395178}, W\l odarczyk \cite{MR2163383}]
\label{hironaka}
 Given a pair $(S,\mc{I})$, where $S$ is a smooth variety and $\mc{I}$ is an ideal sheaf of $\mc{O}_{S}$, then there exists a sequence of blow-ups $\sigma_{i}:S_{i}\to S_{i-1}$ of smooth centers $Z_{i-1}\subset S_{i-1}$,
  $$S_{0}=S\xleftarrow{\sigma_{1}}S_{1}\xleftarrow{\sigma_{2}}S_{2}\xleftarrow{\sigma_{3}}\cdots S_{i}\leftarrow \cdots \xleftarrow{\sigma_{n}}S_{n},$$
  which defined a sequence of pairs $(S_{i},\mc{I}_{i})$ where 
  \begin{itemize}
      \item $\mc{I}_i$ is an ideal sheaf of $S_i$ such that $Z_{i}$ is supported in $\mc{I}_{i}$;
      \item $\mc{I}_{i}\cong \mc{I}(D_{i})^{-1}\sigma_{i}^{*}(\mc{I}_{i-1})$, where $\mc{I}(D_{i})$ is the ideal of the exceptional divisor of $\sigma_{i}$;
      \item $\mc{I}_{n}\cong \mc{O}_{S_{n}}$.
  \end{itemize}
\end{theorem}

\begin{proof}[Proof of \cref{thm:main2}]
  We apply the Hironaka principalization theorem to the pair $(S,\mc{I}(X,S))$. We notice that for any quasi-smooth derived scheme $X_{i}\subset S_{i}$ such that $\mc{I}(X_{i},S_{i})=\mc{I}_{i}$, $Z_{i}$ is a closed derived subscheme of $X_{i}$. Moreover, by \cref{cor:0306}, $\mb{B}l_{Z_{i}/X_{i}}$ is a closed derived subscheme of $S_{i+1}$ and $\mc{I}(\mb{B}l_{Z_{i}/X_{i}},S_{i+1})=\mc{I}_{i+1}$. Thus we can inductively define $X_{i}$ such that $X_{i}:=\mb{B}l_{Z_{i-1}/X_{i-1}}$. Moreover, $X_{n}=\emptyset$ since $\pi_{0}(X_{n})=\emptyset$.
\end{proof}

\section{Intrinsic Blow-ups and Virtual Localization Formula}
\label{sec:C}
Let $U$ be a locally of finite presentation and quasi-compact derived scheme over $\mb{C}$, with an action by a connected reductive $G$ such that the (derived) fixed locus $\pi_{0}(U)^{G}$ is non-empty. In Theorem 5.3 of \cite{hekking2022stabilizer}, Hekking-Rydh-Savvas proved that the derived blow-up $\mb{B}l_{U^{G}}U$ is the derived enhancement of the intrinsic blow-up of $\pi_{0}(U^{G})$ in $\pi_{0}(U)$ in the sense of \cite{kiem2017generalized}. In this section, we will prove that the intrinsic blow-up along the fixed locus will induce a virtual localization formula for the perfect obstruction theory coming from a quasi-smooth derived scheme.

\subsection{The $G$-theory of a derived scheme}
Let $U$ be a quasi-smooth derived scheme over $\mb{C}$. The homotopy category of the $\infty$-category $\mathrm{Coh}(U)$ has a natural $t$-structure such that the heart is $Coh(\pi_{0}(U))$. Hence we have that the $G$-theory
$$G(X):=K(\mathrm{Coh}(X))\cong G(\pi_{0}(U)).$$
The above formula also holds for the equivariant $G$-theory. Now we assume that there is a torus $T$ acts on $X$.
\begin{lemma}[Proposition 3.24 or Proposition 7.4]
  \label{dueto}
The fixed locus $U^{T}$ is also quasi-smooth over $\mb{C}$.
\end{lemma}

Let $g:Y\to U$ be a $T$-equivariant closed embedding such that $(U-Y)^{T}=\emptyset$. Then by  Thomason \cite{thomason1992formule}
\begin{align*}
G^{T}(U-Y)\otimes_{\mathrm{Rep}(T)}Frac(\mathrm{Rep}(T))&\cong G^{T}(\pi_{0}(U)-\pi_{0}(Y))\otimes_{\mathrm{Rep}(T)}Frac(\mathrm{Rep}(T))\\ &\cong 0
\end{align*}
Hence we have the concentration lemma by the localization theorem
\begin{lemma}[Concentration lemma]
  \label{lem:concentration}
  The embedding $g$ induces an isomorphism
  $$g_{*}:G^{T}(Y)\otimes_{\mathrm{Rep}(T)}Frac(\mathrm{Rep}(T)) \cong G^{T}(U)\otimes_{\mathrm{Rep}(T)}Frac(\mathrm{Rep}(T)).$$
\end{lemma}
Specifically, we denote $i_{T}:U^{T}\to U$ as the closed embedding of fixed locus. Then we have the isomorphism
$$i_{T*}:G^{T}(U^{T})\otimes_{\mathrm{Rep}(T)}Frac(\mathrm{Rep}(T))\cong G^{T}(U)\otimes_{\mathrm{Rep}(T)}Frac(\mathrm{Rep}(T)).$$

\subsection{The infinite wedge sum $[\frac{1}{\wedge^{\bullet}\mc{F}}]$}
Let $Y$ be a quasi-smooth derived scheme over $\mb{C}$, with a trivial torus $T$ action. Let $\mc{F}$ be $T$-equivariant quasi-coherent sheaf of perfect-amplitude contained in $[0,1]$. Moreover, we assume that for every closed point $y\in Y$, the $T$-action on $\mc{F}|_{y}$ has no trivial sub-representations. The purpose of this subsection is to define
$$[\frac{1}{\wedge^{\bullet}\mc{F}}]\in G_{0}^{T}(Y)\otimes_{\mathrm{Rep(T)}}Frac(\mathrm{Rep}(T)).$$

The torus $T$ action on $\mc{F}$ induces an action of $T$ on $\mb{V}_{\mb{P}_{Y}(\mc{F})}(\mc{O}_{\mb{P}_{Y}(\mc{F})}(1))$. By  \cref{prop:conj41}, we have
$$\mb{V}_{\mb{P}_{Y}(\mc{F})}(\mc{O}_{\mb{P}_{Y}(\mc{F})}(1))-\mb{P}_{Y}(\mc{F})\cong \mb{V}^{*}_{X}(\mc{F})$$
and $(\mb{V}_{Y}^{*}(\mc{F}))^{T}\cong \emptyset$. Thus by the concentration lemma \cref{lem:concentration}, there exists a unique class, which we denote as $$[\frac{1}{1-\mc{O}_{\mb{P}_{Y}}(1)}]\in G_{0}^{T}(\mb{P}_{Y}(\mc{F}))\otimes_{\mathrm{Rep(T)}}Frac(\mathrm{Rep}(T))$$
such that
$$i_{0*}[\frac{1}{1-\mc{O}_{\mb{P}_{Y}}(1)}]\cong [\mc{O}_{\mb{V}_{\mb{P}_{Y}(\mc{F})}(\mc{O}_{\mb{P}_{Y}(\mc{F})}(1))}].$$
It is the inverse of $1-[\mc{O}_{\mb{P}_{Y}}(1)]$ in $G_{0}^{T}(\mb{P}_{Y}(\mc{F}))\otimes_{\mathrm{Rep(T)}}Frac(\mathrm{Rep}(T))$.
\begin{definition}
  \label{def:virtually}
  We define the class
  $$[\frac{1}{\wedge^{\bullet}\mc{F}}]\in G_{0}^{T}(Y)\otimes_{\mathrm{Rep(T)}}Frac(\mathrm{Rep}(T))$$
  as
  $$pr_{\mc{F}*}(\frac{1}{1-[\mc{O}_{\mb{P}_{Y}}(1)]})+(-1)^{r}det(\mc{F})^{-1}\sum_{l=0}^{-r}[S^{l}(\mc{F}^{\vee})]$$
  where $r$ is the rank of $\mc{F}$.
\end{definition}
The \cref{def:virtually} is compatible with the traditional definition, due to the following lemma:
\begin{lemma}
  \label{classicallocal}
  Let $g:S\to Y$ be a map of derived schemes, such that $g^{*}\mc{F}\cong \{V\to W\}$, where $V,W$ are $T$-equivariant locally free sheaves over $Y$ such that there $V,W$ has no components with trivial representations, and $r=w-v$. Then we have
    $$[\frac{1}{\wedge^{\bullet}(g^{*}\mc{F})}]\cong [\wedge^{\bullet}V][\frac{1}{\wedge^{\bullet}W}].$$
\end{lemma}
\begin{proof}
  
  We can assume that $g=id$, as the projectivization is compatible with pull-backs. We denote $v$ and $w$ as the ranks of $V$ and $W$ respectively.

  Over $\mb{P}_{Y}(\mc{F})$, we have the following fiber sequences:
  \begin{equation*}
    L_{\mb{P}_{Y}(W)/Y}|_{\mb{P}_{Y}(\mc{F})}\otimes \mc{O}_{\mb{P}_{Y}(\mc{F})}(1)\to pr_{\mc{F}}^{*}W\to \mc{O}_{\mb{P}_{Y}(\mc{F})}(1)
  \end{equation*}
  and thus we have
  \begin{equation*}
    [\frac{1}{1-\mc{O}_{\mb{P}_{Y}(\mc{F})}(1)}]\cdot pr^{*}_{\mc{F}}[\wedge^{\bullet}W]=\sum_{i=0}^{w}(-1)^{i}\sum_{j=0}^{i}(-1)^{j}[\wedge^{i-j}pr_{\mc{*}}W][\mc{O}_{\mb{P}_{Y}(\mc{F})}(j)].
  \end{equation*}
 We first prove the case that $r>0$: by Jiang's generalized Serre theorem \cref{thm:serre2}, we have
  \begin{align*}
    & pr_{\mc{F}*}[\frac{1}{1-\mc{O}_{\mb{P}_{Y}(\mc{F})}(1)}]\cdot [\wedge^{\bullet}W]\\ &=\sum_{i=0}^{w}(-1)^{i}\sum_{j=0}^{i}(-1)^{j}[\wedge^{i-j}W][S^{j}\mc{F}] \\
    &=[\wedge^{\bullet}V].
  \end{align*}
  Next we prove the case $r\leq 0$ be the induction on $-r$. Over $\mb{P}_{Y}(\mc{F}\oplus t\mc{O})$, where $t$ is a non-trivial representation, the derived stack  $\mb{P}_{Y}(\mc{F})$ is the derived locus of
  $$t\mc{O}_{\mb{P}_{Y}(\mc{F})}(-1)\to \mc{O}_{\mb{P}_{Y}(\mc{F})}$$
  and thus we have
  \begin{align*}
    pr_{\mc{F}*}[\frac{1}{1-\mc{O}_{\mb{P}_{Y}(\mc{F})}(1)}]=pr_{(\mc{F}\oplus t\mc{O})*}[\frac{1}{1-\mc{O}_{\mb{P}_{Y}(\mc{F}\oplus t\mc{O})}(1)}]-pr_{(\mc{F}\oplus t\mc{O})*}[\frac{t\mc{O}_{\mb{P}_{Y}(\mc{F}\oplus t\mc{O})}(-1)}{1-\mc{O}_{\mb{P}_{Y}(\mc{F}\oplus t\mc{O})}(1)}]\\
    =(1-t)pr_{(\mc{F}\oplus t\mc{O})*}[\frac{1}{1-\mc{O}_{\mb{P}_{Y}(\mc{F}\oplus t\mc{O})}(1)}]+t\cdot pr_{(\mc{F}\oplus t\mc{O})*}\mc{O}_{\mb{P}_{Y}(\mc{F}\oplus t\mc{O})}(-1).
  \end{align*}
  Thus \cref{classicallocal} follows from the induction on $\mb{P}_{Y}(\mc{F}\oplus t\mc{O})$, since the rank of $\mc{F}\oplus t\mc{O}$ is $r+1$.
\end{proof}

\subsection{The virtual localization theorem}
Now we consider the virtual localization theorem. Let $U$ be a quasi-smooth derived scheme with a constant virtual dimension over $\mb{C}$ with a torus $T$ action. Let $V_{1},\cdots, V_{l}$ be the connected components of $U^{T}$, which are all quasi-smooth due to \cref{dueto}. As the rank of $C_{U^{T}/U}$ is a locally constant, all $C_{V_{j}/U}$ has constant ranks which we denote as $r_{j}$.
\begin{theorem}[Virtual localization theorem]
  We have
  $$[\mc{O}_{U}]=\sum_{j=1}^{l}i_{V_{j}*}[\frac{1}{\wedge^{\bullet}C_{V_{j}/U}}]\in G_{0}^{T}(U)\otimes_{\mathrm{Rep(T)}}Frac(\mathrm{Rep}(T)).$$
\end{theorem}
\begin{proof}
  We consider the intrinsic blow-up $\mb{B}l_{U^{T}}U$. Then by \cref{eq:noref}, we have
  $$[\mc{O}_{U}]=[pr_{U^{T}/U*}\mc{O}_{\mb{B}l_{U^{T}/U}}]+\sum_{j=1}^{l}i_{V_{j}*}(-1)^{r_{j}}det(C_{V_{j}/U})^{-1}\sum_{k=0}^{-r_{i}}[S^{k}(C_{V_{j}/U}^{\vee})].$$
  Moreover, we consider the exceptional divisors $\eta_{j}:E_{j}\to \mb{B}l_{U^{T}}U$, where $E_{j}\cong \mb{P}_{V_{j}}(C_{V_{j}/U})$. Then $(\mb{B}l_{U^{T}}U-\cup_{j=1}^{l}E_{j})^{T}=\emptyset$ and thus we have the isomorphism:
  $$\bigoplus_{j=1}^{l}\eta_{j*}:\bigoplus_{j=1}^{l}G_{0}(E_{j})\otimes_{\mathrm{Rep(T)}}Frac(\mathrm{Rep}(T))\to G_{0}(\mb{B}l_{U^{T}}U)\otimes_{\mathrm{Rep(T)}}Frac(\mathrm{Rep}(T)).$$
  By restricting to $E_{j}$, we have
  $$[\mc{O}_{\mb{B}l_{U^{T}/U}}]\cong \sum_{j=1}^{l}\eta_{j*}[\frac{1}{1-\mc{O}_{\mb{P}_{V_{j}(C_{V_{j}/U})}}(1)}].$$
  in $ G_{0}(\mb{B}l_{U^{T}}U)\otimes_{\mathrm{Rep(T)}}Frac(\mathrm{Rep}(T))$.
  Thus we have $[\mc{O}_{U}]$ is the sum of all
  $$i_{V_{j}*}(pr_{C_{V_{j}/U}*}\frac{1}{1-[\mc{O}_{\mb{P}_{V_{j}}(C_{V_{j}/U})}(1)]}+(-1)^{r_{j}}det(C_{V_{j}/U})^{-1}\sum_{k=0}^{-r_{j}}[S^{k}(C_{V_{j}/U}^{\vee})])$$
  for all $1\leq j\leq l$, which is
  $$\sum_{j=1}^{l}i_{V_{j}*}[\frac{1}{\wedge^{\bullet}C_{V_{j}/U}}]\in G_{0}^{T}(U)\otimes_{\mathrm{Rep(T)}}Frac(\mathrm{Rep}(T)).$$
\end{proof}

\begin{remark}
  To generalize the above proof to Artin stacks, we would need a concentration theorem, like Theorem A of \cite{aranha2022localization}. It will appear in the future work of Aranha-Khan-Latyntsev-Park-Ravi, by Remark 0.4 of \cite{aranha2022localization}.
\end{remark}
\bibliography{Resolution.bib}
\bibliographystyle{plain}
\vspace{5mm}
Kavli Institute for the Physics and 
Mathematics of the Universe (WPI), University of Tokyo,
5-1-5 Kashiwanoha, Kashiwa, 277-8583, Japan.

\textit{E-mail address}: yu.zhao@ipmu.jp
\end{document}